\documentclass[a4,english,10pt]{amsart}
\usepackage[T1]{fontenc}
\usepackage[utf8]{inputenc}
\usepackage{babel}

\usepackage{graphicx}

\usepackage{amsmath,amsfonts,enumerate,amssymb}
\usepackage{amsthm}

\usepackage[usenames,dvipsnames,svgnames]{xcolor}

\theoremstyle{plain}
\newtheorem{thm}{Theorem}[section]

\theoremstyle{plain}
\newtheorem{prop}[thm]{Proposition}
\newtheorem{lem}[thm]{Lemma}
\newtheorem{cor}[thm]{Corollary}

\theoremstyle{definition}
\newtheorem{rem}[thm]{Remark}

\theoremstyle{example}

\newcommand*{\B}{\mathcal{B}}
\newcommand*{\I}{\mathcal{I}}
\newcommand*{\J}{\mathcal{J}}
\newcommand*{\T}{\mathcal{T}}
\newcommand*{\al}{\alpha}
\newcommand*{\card}{\#}

\newcommand*{\ZZ}{\mathbb{Z}}

\newtheorem{lemma}[thm]{Lemma}

\DeclareMathOperator{\rint}{rint}

\newcommand*{\uR}{\underline{\mathbb{R}}}

\newcommand*{\oS}{\overline{S}}
\newcommand*{\ve}{\varepsilon}
\newcommand*{\de}{\delta}

\newcommand*{\ff}{\varphi}

\newcommand*{\RR}{{\mathbb{R}}}

\newcommand*{\CC}{{\mathbb{C}}}

\def\car{\mathbf{1}}

\newcommand*{\bR}{\mathbb{R}}

\newcommand*{\zz}{\mathbf{z}}
\newcommand*{\vv}{\mathbf{v}}
\newcommand*{\uu}{\mathbf{u}}

\newcommand*{\mv}{\mathbf{m}}

\newcommand*{\ww}{\mathbf{w}}
\newcommand*{\mol}{\overline{m}}
\newcommand*{\mul}{\underline{m}}

\newcommand{\ol}[1]{\overline{#1}}

\newcommand*{\yy}{\mathbf{y}}
\newcommand*{\ee}{\mathbf{e}}

\newcommand*{\xx}{\mathbf{x}}

\newcommand*{\TT}{\mathbb{T}}
\newcommand*{\NN}{\mathbb{N}}
\def\downto{\downarrow}
\def\upto{\uparrow}

\def\calX{\mathcal{X}}

\def\veps{\varepsilon}

\newenvironment{abc}{\begin{enumerate}[{\rm (a)}]}{\end{enumerate}}
\newenvironment{num}{\begin{enumerate}[{\rm 1.}]}{\end{enumerate}}

\def\arxivfn{The numbering refers to the current arxiv version.}

\long\def\blue#1{{\color{blue}#1}}

\def\oSS{\oS}

\begin{document}

\title[Fenton type minimax problems]{Fenton type minimax problems for \\ sum of translates functions}

\author{Bálint Farkas, Béla Nagy and Szilárd Révész}

\date{}

\begin{abstract}
Following an insightful work of P.{} Fenton, we investigate sum of translates functions $F(\xx,t):=J(t)+\sum_{j=1}^n \nu_j K(t-x_j)$, where $J:[0,1]\to \uR:=\RR\cup\{-\infty\}$ is a ``sufficiently non-degenerate'' and upper-bounded ``field function'', and $K:[-1,1]\to \uR$ is a fixed ``kernel function'', concave both on $(-1,0)$ and $(0,1)$, $\xx:=(x_1,\ldots,x_n)$ with $0\le x_1\le\dots\le x_n\le 1$, and $\nu_1,\dots,\nu_n>0$ are fixed. We analyze the behavior of the local maxima vector $\mv:=(m_0,m_1,\ldots,m_n)$, where $m_j:=m_j(\xx):=\sup_{x_j\le t\le x_{j+1}} F(\xx,t)$, with $x_0:=0$, $x_{n+1}:=1$; and study the optimization (minimax and maximin) problems $\inf_{\xx}\max_j m_j(\xx)$ and $\sup_{\xx}\min_j m_j(\xx)$.  The main result is the equality of these quantities, and provided $J$ is upper semicontinuous, the existence of extremal configurations and their description as equioscillation points $\ww$, i.e., $\ww$ satisfying $m_0(\ww)=m_1(\ww)=\cdots=m_n(\ww)$.  In our previous papers we obtained results for the case of singular kernels, i.e., when $K(0)=-\infty$ and the field was assumed to be upper semicontinuous. In this work  we get rid of these assumptions and prove  common generalizations of Fenton's and our previous results, and arrive at the greatest generality in the setting of concave kernel functions.

\smallskip


\end{abstract}

\maketitle

\section{Introduction}\label{sec:Introduction}
The very notion of the sum of translates functions originates from an ingenious paper of Fenton \cite{Fenton}, who himself worked out results on them for use in his work \cite{FentonEnt} proving a conjecture of P.D.{} Barry. Later he found other applications of his method, see \cite{FentonCos, FentonCos2}. Inspired by Fenton, in \cite{TLMS2018, Homeo,Minimax} we analyzed interval maxima vectors of sum of translates functions and found variants of Fenton's minimax results.
For the origins and wide range of applications, of the approach, ranging from the strong polarization problem to moving node Hermite-Fejér interpolation, Chebyshev constants and Bojanov theorems we refer the reader to the papers \cite{Homeo,Minimax, TLMS2018,Ural}. Before going to the results we present the setting of the problem in detail.

A function $K:(-1,0)\cup (0,1)\to \bR$ will be called a \emph{kernel function} if it is concave on $(-1,0)$ and on $(0,1)$, and if it satisfies
\begin{equation}\label{eq:Kzero}
\lim_{t\downto 0} K(t) =\lim_{t\upto 0} K(t).
\end{equation}
By the concavity assumption these limits exist, and a kernel function has one-sided limits also at $-1$ and $1$. We set
\begin{equation*}
K(0):=\lim_{t\to 0}K(t),\quad K(-1):=\lim_{t\downto -1} K(t) \quad\text{and}\quad K(1):=\lim_{t\upto 1} K(t).
\end{equation*}
We note explicitly that we thus obtain the extended continuous function $K:[-1,1]\to \bR\cup\{-\infty\}=:\uR$, and that we have $\sup K<\infty$. Also note that a kernel function is almost everywhere differentiable.

A kernel function $K$ is called \emph{singular} if
\begin{equation}\label{cond:infty}
\tag{$\infty$}
K(0)=-\infty.
\end{equation}

We say that the kernel function $K$ is \emph{strictly concave} if it is strictly concave on both of the intervals $(-1,0)$ and $(0,1)$.

Further, we call it \emph{monotone}\footnote{These conditions---and more, like $C^2$ smoothness and strictly negative second derivatives---were assumed on the kernel functions in the ground-breaking paper of Fenton \cite{Fenton}.} if
\begin{equation}
\label{cond:monotone}\tag{M}
K \text{ is non-increasing on } (-1,0) \text{ and non-decreasing on } (0,1).
\end{equation}
By concavity, under the monotonicity condition \eqref{cond:monotone} the endpoint values $K(-1), K(1)$ are also finite. If $K$ is strictly concave, then \eqref{cond:monotone} implies \emph{strict monotonicity}:
\begin{equation}
\label{cond:smonotone}\tag{SM}
K \text{ is strictly decreasing on } [-1,0) \text{ and strictly increasing on } (0,1],
\end{equation}
where we have extended the assertion to the finite endpoint values, too.

Let $n\in \NN=\{1,2,\dots,\}$ be fixed. We will call a function $J:[0,1]\to\uR$ an \emph{(external) $n$-field function}\footnote{Again, the terminology of kernels and fields came to our mind by analogy, which in case of the logarithmic kernel $K(t):=\log|t|$ and an external field $J(t)$ arising from a weight $w(t):=\exp(J(t))$ are indeed discussed in logarithmic potential theory. However, in our analysis no further potential theoretic notions and tools will be applied. This is so in particular because our analysis is far more general, allowing different and almost arbitrary kernels and fields; yet the resemblance to the classical settings of logarithmic potential theory should not be denied.}, or---if the value of $n$ is unambiguous from the context---simply a \emph{field} or \emph{field function}, if it is bounded above  on $[0,1]$, and it assumes finite values at more than  $n$ different points, where we count the points $0$ and $1$ with weight\footnote{The weighted counting makes a difference only for the case when $J^{-1}(\{-\infty\})$ contains the two endpoints; with only $n-1$ further interior points in $(0,1)$ the weights in this configuration add up to $n$ only, hence such $J$ are considered inadmissible for being an $n$-field function.} $1/2$ only, while the points in $(0,1)$ are accounted for with weight $1$.
Furthermore, for a field function $J$ we define
its \emph{singularity set} $X$, with complement $X^c$ called the \emph{finiteness domain}, by
\begin{equation}\label{eq:Xdef}
X:=X_J:=J^{-1}(\{-\infty\})\quad\text{and}\quad X^c:=[0,1]\setminus X=J^{-1}(\RR).
\end{equation}
Then $X^c$ has cardinality exceeding $n$ (in the above described, weighted sense), in particular $X\neq [0,1]$.
Let us point out that striving for generality in regard of $J$, in particular as for the almost arbitrariness of the finiteness domain $X^c$, is important in various applications. In \cite{Homeo} we discussed applications where the interpolation type results required to have $X^c$ minimal (i.e. $n+1$ points only) and also applications to Chebyshev constants of compact, perfect sets $E\subset [0,1]$, corresponding to the case $J=\log\car_E$ and $X^c=E$.

\medskip  Throughout this work we consider the \emph{open simplex}
\[
S:=S_n:=\{\yy : \yy=(y_1,\dots,y_n)\in (0,1)^n,\: 0< y_1<\cdots <y_n<1\},
\]
and its closure the \emph{closed simplex}
\[
\overline{S}:=\{\yy: \yy\in [0,1]^n,\: 0\leq y_1\leq \cdots \leq y_n\leq 1\},
\]
whose elements $\yy=(y_1,\dots,y_n)$  will be called node systems, and $y_1,\dots, y_n$ nodes. We use the norm $\|\xx\|:=\max\{|x_1|,\dots,|x_n|\}$.

\medskip
For given $n\in \NN$, a kernel function $K$, constants $\nu_j>0 ~ (j=1,\ldots,n)$  we define the \emph{pure sum of translates function}
\begin{equation}\label{eq:puresum}
f(\yy,t):=\sum_{j=1}^n \nu_j K(t-y_j)\quad (\yy\in \overline{S},\: t\in [0,1]),
\end{equation}
and for a further given $n$-field function $J$ also the \emph{(weighted) sum of translates function}
\begin{equation}\label{eq:Fsum}
F(\yy,t):=J(t)+\sum_{j=1}^n \nu_jK(t-y_j)=J(t)+f(\yy,t)\quad (\yy\in \overline{S},\: t\in [0,1]).
\end{equation}
More generally, if $K_1,\dots, K_n:[-1,1]\to \uR$ and $J:[0,1]\to \uR$ are given functions (not necessarily kernels) the corresponding \emph{generalized sum of translates function} is defined as
\begin{equation}\label{eq:Fsumgen}
F(\yy,t):=J(t)+\sum_{j=1}^n K_j(t-y_j)\quad (\yy\in \overline{S},\: t\in [0,1]).
\end{equation}

Note that the functions $J, K$ can take the value $-\infty$, but not $+\infty$, therefore the sum of translates functions can be defined meaningfully.
Furthermore, $f:\oS \times [0,1] \to \uR$ is  continuous in the extended sense; and it is also continuous (in the usual sense) provided $K$ is non-singular. The  function $F(\yy,\cdot)$ is not constant $-\infty$, hence\footnote{These require some careful considerations and the assumed degree of non-singularity of $J$ is in fact \emph{the exact condition} to ensure $F \not\equiv -\infty$. For details see \cite{Homeo,Minimax}.} $\sup_{t\in[0,1]}F(\yy,t)>-\infty$ for every $\yy\in \oS$.
Moreover, $F$ is upper semicontinuous on $\oS\times [0,1]$ provided $J$ has this property.

The \emph{singularity set} of $F(\yy,\cdot)$ is
\[
\widehat{X}(\yy):=\{t\in[0,1]~:~F(\yy,t)=-\infty\} \varsubsetneq [0,1].
\]
Accordingly, an interval $I\subseteq [0,1]$ with $I\subseteq \widehat{X}(\yy)$ will be called \emph{singular}.  In case $K$ (hence $f$) is finite valued (e.g., when $K$ is non-singular and monotone), we necessarily have $\widehat{X}(\yy)=X_J$ for each $\yy\in \oS$, and the notion of singularity of intervals, hence of node systems, becomes totally independent of the kernel $K$ itself.

\medskip Writing $y_0:=0$ and $y_{n+1}:=1$  we set for each $\yy\in \overline{S}$ and $j\in \{0,1,\dots, n\}$
\begin{align*}
I_j(\yy)&:=[y_j,y_{j+1}],
\qquad m_j(\yy):=\sup_{t\in I_j(\yy)} F(\yy,t),
\end{align*}
and
\begin{align*}
\mol(\yy)&:=\max_{j=0,\dots,n} m_j(\yy)=\sup_{t\in [0,1]}F(\yy,t),
\qquad \mul(\yy):=\min_{j=0,\dots,n} m_j(\yy).
\end{align*}
Further, we also define
\begin{equation*}
	M(\oSS):=\inf_{\yy\in \oS} \mol(\yy)\quad\text{and}\quad m(\oSS):=\sup_{\yy\in \oS} \mul(\yy).
\end{equation*}

As has been said above, for each $\yy\in \oS$ we have that $\mol(\yy)=\sup_{t \in [0,1]} F(\yy,t) \in \RR$ is finite.
Observe that an interval $I\subseteq [0,1]$ is contained in $\widehat{X}(\yy)$, i.e., $I$ is singular, if and only if $F(\yy,\cdot)|_I\equiv -\infty$. In particular $m_j(\yy)=-\infty$ exactly when $I_j(\yy)\subseteq \widehat{X}(\yy)$. A node system $\yy$ is called \emph{singular} if there is $j\in \{0,1,\dots,n\}$ with $I_j(\yy)$ singular, i.e., $m_j(\yy)=-\infty$; and a node system $\yy\in \partial S= \oS\setminus S$ is called \emph{degenerate}.

An essential role is played by the \emph{regularity set}
\begin{align}\label{eq:Ydef}
Y&:=\{\yy\in \oS: \text{$\yy$ is non-singular}\}\notag \\
&=\{\yy\in \oS: \text{$I_j(\yy)\not\subseteq \widehat{X}(\yy)$ for $j=0,1,\dots,n$}\}\\
&=\{\yy\in \oS: \text{$m_j(\yy)\neq-\infty$ for $j=0,1,\dots,n$}\}.\notag
\end{align}
If the kernel $K$ is singular, then all degenerate node systems are  singular, hence $Y\subset S$. Note also that we have $S\subset Y$ if and only if $X$ has empty interior. So if $K$ is singular and $X$ has empty interior, then $Y=S$.

We also introduce the \emph{interval maxima vector function}
\[
\mv(\ww):=(m_0(\ww),m_1(\ww),\ldots,m_n(\ww)) \in \uR^{n+1} \quad (\ww \in \oS).
\]
From the above it follows that for $\ww\in \oS$ we have $\mv(\ww)\ne (-\infty,\ldots,-\infty)$.

To recall the fundamental result of Fenton, we need to present another assumption, introduced by Fenton \cite{Fenton}: The ``cusp condition'' requires that ($K$ is differentiable and) $\lim_{t \upto 0} K'(t)=-\infty$, $\lim_{t \downto 0} K'(t)=+\infty$. In order to extend investigations to not necessarily differentiable functions, we reformulate this condition in the following way:
\begin{equation}\label{cond:inftyprime}
\tag{$\infty'_\pm$}
\lim_{t,x\upto 0}\frac{K(t)-K(x)}{t-x} =-\infty
\qquad \textrm{ and } \qquad
\lim_{t,x\downto 0}\frac{K(t)-K(x)}{t-x}=\infty.
\end{equation}
Observe that the singularity  condition \eqref{cond:infty} implies, by concavity, \eqref{cond:inftyprime}, too. In the above setting Fenton's result takes the following form.

\begin{thm}[Fenton]\label{th:Fenton}
Let the field $J:(0,1)\rightarrow \bR$ be concave, and the kernel $K$ be monotone \eqref{cond:monotone}, strictly concave and $C^2$, with $K''<0$ on $(-1,0)\cup(0,1)$ and satisfying the cusp condition \eqref{cond:inftyprime}.

Take $\nu_1=\ldots=\nu_n$. Then for the sum of translates function defined in \eqref{eq:Fsum}  there exists a unique
minimax node system $\ww:=(w_1,\ldots,w_n)$ in the open simplex $S$:
\begin{equation}\label{eq:minimaxpt}
M(\oSS)=\mol(\ww).
\end{equation}
Moreover, $F(\ww,\cdot)$ equioscillates on the intervals $I_j(\ww)=[w_j,w_{j+1}]$, meaning that
\begin{equation}\label{eq:defekvp}
m_j(\ww):=\sup_{[w_j,w_{j+1}]}F(\ww,\cdot)=\sup_{[w_i,w_{i+1}]}F(\ww,\cdot)=:m_i(\ww) \quad (0\le i,j\le n),
\end{equation}
the point $\ww$ is the unique point with this property,  and it is also the unique maximin point with $\mul(\ww)=m(\oSS)$.
\end{thm}
A point $\ww\in\oS$ satisfying \eqref{eq:defekvp} is called  an \emph{equioscillation point}.

The above result of Fenton was extended in various directions. On the torus $\TT:=\RR/\ZZ$, a development started with a particular result and a conjecture by Ambrus, Ball and Erdélyi \cite{AmbrusBallErdelyi}---for the progress in this direction we mention the work of Erdélyi, Hardin, Saff \cite{ErdelyiHardinSaff}  and Hardin, Kendall and Saff \cite{HardinKendallSaff}. In \cite{TLMS2018} we extended the investigation to translates of not necessarily equal kernels, which is also mirrored in our above setup using the more general terms $K_j=\nu_j K$ in \eqref{eq:puresum} and \eqref{eq:Fsum}. Dealing with different kernels was crucial in tackling, e.g., the application to Bojanov type theorems of approximation theory, see \cite{TLMS2018}, \cite{Minimax}.

More appropriate for us here is, however, the setup on the real interval $[0,1]$, matching the setup of Fenton. It challenged us for long what conditions are really necessary and how results can be generalized (and then applied). However, until recently we were working with a stronger condition than Fenton in regard of the singularity assumption. We had good reasons for that: it not only made some intricate proofs possible, but also led to a notable and unexpected result, that of the so-called ``homeomorphism theorem'' \cite[Thm.{} 2.1]{Homeo}, which simply cannot hold true if the kernels $K$ are not singular. We do not formulate it here, however, because our focus is different here. Now we just want to return to the original setup of Fenton, and generalize it in the true sense, that is, not assuming more about singularity at $0$ than he did.

With this we will succeed, however, only by a multiple and fundamental exploitation of our results for the singular case. Therefore, even if not discussing the ways we obtained those results, but we have to recall the end outcomes.

The following result follows from a combination of Corollary 3.2 and Corollary 4.2 of \cite{Minimax}.
\begin{thm}\label{thm:mainold}
	Let $n\in\NN$, $\nu_1,\ldots,\nu_n>0$, $K$ be a singular \eqref{cond:infty} and monotone \eqref{cond:monotone} kernel function, and let $J$ be an upper semicontinuous $n$-field function.

Then $M(\oSS)=m(\oSS)$ and there exists some node system $\ww\in \oS$, also belonging to $Y$, with the three properties that it is an equioscillation point, it attains the simplex maximin and also it attains the simplex minimax: $\mul(\ww)=m(\oSS)=M(\oSS)=\mol(\ww)$.

Moreover, strict majorization $m_i(\xx)>m_i(\yy)$ for every $i=0,1,\ldots,n$ cannot hold for any $\xx, \yy \in Y$.
\end{thm}
The situation that no pairs $\xx,\yy\in Y$ exist with $m_j(\xx)>m_j(\yy)$ for every $j=0,1,\ldots,n$, will be expressed by saying that \emph{strict majorization does not hold} on $Y$.

The main result of the paper is the following common generalization of Theorem \ref{thm:mainold} and Theorem \ref{th:Fenton}, dropping basically all assumptions.
\begin{thm} \label{thm:main}
	Let $n\in\NN$, $\nu_1,\ldots,\nu_n>0$, $K$ be a \emph{monotone} \eqref{cond:monotone} kernel function and $J$ be an arbitrary $n$-field function.
Then $M(\oSS)=m(\oSS)$ and there exists some node system $\ww\in \oS$ such that
\[
\mol(\ww)=M(\oSS)=\min_{\xx\in \oS}\mol(\xx).
\]
Further, for any equioscillation point $\ee$ we have $\mol(\ee)=M(\oSS)$, and moreover, if $J$ is upper semicontinuous or $K$ is singular then there exists, in fact, an equioscillation point.

Furthermore, strict majorization $m_i(\xx)>m_i(\yy)$ for every $i=0,1,\ldots,n$ cannot hold for any $\xx, \yy \in Y$.
\end{thm}
Concerning Bernstein-Erd\H os-type equioscillation characterization for some minimax problems Shi presented in \cite{Shi} an abstract approach in the framework of differentiable functions. The conditions there involve assumption on the Jacobian of the interval maxima  vector function $\mv$ and also a singularity type condition. These are not fulfilled in our situation, as the here occuring functions have (in general) very weak continuity properties and the singularity plays no role whatsoever.

We start with some auxiliary results and then  prove the Theorem \ref{thm:main} in the following steps.
\begin{num}
	\item We prove that majorization cannot hold on $Y$ for $J$ extended continuous.
	\item We prove that majorization cannot hold on $Y$ for $J$ upper semicontinuous.
	\item We show,  for upper semicontinuous  $J$, the existence of an equioscillation point and the minimax/maximin result.
	\item We handle the case of general $J$ when $K$ is singular.
	\item We finish the proof by tackling the case of general $J$ when $K$ is non-singular.
	\end{num}

\section{Technical lemmas}\label{sec:Prelimlem}

We collect here results of general nature that are needed for the proof Theorem \ref{thm:main}.

\begin{rem}\label{rem:nonmaj0}
If there is an equioscillation point, then $M(\oSS)\leq m(\oSS)$. Indeed, if $\ee\in \oS$ is an equioscillation point, then
\[
M(\oSS)\leq\mol(\ee)=\mul(\ee)\leq m(\oSS).
\]
The converse implication is not true is $J$ not upper semicontinuous.
Consider the case $K\equiv 0$, $n=1$ and $J(t)=\car_{[0,1/2)}(t)t$. Then there is no equioscillation point, but of course $M(\oSS)=m(\oSS)=1/2$. (There is however an equioscillation point if one considers the upper semicontinuous regularization of $J^*(t)=\car_{[0,1/2]}(t)t$ instead. This will be important below.)
\end{rem}

The following easy lemma will be helpful.
\begin{lem}\label{lem:eqmS}
Let $n\in\NN$, $\nu_1,\ldots,\nu_n>0$,
let $K$ be a kernel function and let $J$ be an $n$-field function.
Assume that there exists an equioscillation point $\ee\in \oS$ with value $\mu$.
Furthermore, suppose that strict majorization does not hold
on $Y$, i.e.,
there are no $\xx,\yy\in Y$ satisfying $m_j(\xx)>m_j(\yy)$ for every $j\in \{0,1,\dots,n\}$.
Then $m(\oSS)=\mu$.
	\end{lem}
	\begin{proof}
	We obviously have $m(\oSS)\geq \mul(\ee)=\mu>-\infty$.
	Suppose for a contradiction that $m(\oSS)>\mu$. Take $\ww\in \oS$ with $\mul(\ww)>\mu$. Then $\ww\in Y$, and we have
	\[
	m_j(\ww)>\mu=m_j(\ee)\quad\text{for each $j=0,1,\dots,n$},
	\]
	contradicting  that majorization does not hold on $Y$, and we are done.
	\end{proof}

Let us record the following, trivial but extremely useful fact as a separate lemma.

\begin{lemma}[\textbf{Trivial lemma}]\label{lem:trivi}
Let $f, g, h :D \to \uR$ be upper semicontinuous functions on some Hausdorff topological space $D$ and
let $\emptyset\neq A \subseteq B \subseteq D$ be arbitrary.
Assume
\begin{equation}\label{eq:trivia}
f(t)<g(t) \qquad\text{for all $t\in A$}.
\end{equation}
If $A\subseteq B$ is a compact set, then
\begin{equation}\label{eq:trivia-plus}
\max_A (f+h) < \sup_B (g+h) \qquad {\rm unless} \qquad h\equiv-\infty \quad\text{on }\quad A.
\end{equation}
\end{lemma}
\begin{proof} For a proof see Lemma 3.2 in \cite{Minimax}. \end{proof}

\begin{lemma}[\textbf{Interval perturbation lemma}]\label{lem:widening}
Let $K$ be any kernel function. Let $0<\alpha<a<b<\beta< 1$ and  $p, q >0$. Set
\begin{equation}\label{eq:mudef}
\kappa:=\frac{p(a-\alpha)}{q(\beta-b)}.
\end{equation}
\begin{abc}
\item If $K$ satisfies \eqref{cond:monotone} and $\kappa\geq 1$, then for every $t\in [0,\alpha]$ we have
\begin{equation}\label{eq:wideninglemma}
pK(t-\alpha)+qK(t-\beta) \le pK(t-a)+qK(t-b).
\end{equation}
\item If $K$ satisfies \eqref{cond:monotone} and $\kappa\leq 1$, then \eqref{eq:wideninglemma} holds for every $t\in [\beta,1]$.
\item If $\kappa=1$ (and even if $K$ does not necessarily satisfy \eqref{cond:monotone}), then \eqref{eq:wideninglemma} holds for every $t\in [0,\alpha] \cup [\beta,1]$.
\item In case of a strictly concave kernel function {\upshape(a)}, {\upshape(b)} and {\upshape(c)} hold with strict inequality in \eqref{eq:wideninglemma}.
\item If $K$ satisfies \eqref{cond:monotone}, then for every $t\in [a,b]$
\begin{equation}\label{eq:wideninginside}
pK(t-\alpha)+qK(t-\beta) \geq  pK(t-a)+qK(t-b),
\end{equation}
with strict inequality if $K$ is strictly monotone.
\end{abc}
\end{lemma}
\begin{proof} The essence of this lemma was well-known in various forms, the earliest explicit occurrence we could trace is in \cite{Rankin}. For a proof of exactly this  form and some remarks about the context, see Lemma 3.1 of\footnote{\arxivfn} \cite{Minimax}.
\end{proof}

\section{Continuity-type technical results}\label{sec:conttech}

Obviously, the pure sum of translate functions are (extended) continuous, while the sum of translates functions are continuous only partially in their first variable $\xx\in S$. If $K$ is singular, then also the functions $m_j$ are extended continuous, which is not immediately obvious due to the arbitrariness of $J$, but was proven in \cite{Homeo} as Lemma 3.3. We recall this here, too\footnote{\arxivfn}.

\begin{lemma}\label{lem:mjcont}
If the kernel function $K$ is singular, and if $J$ is an arbitrary field function, then for each $j\in \{0,1,\dots, n\}$ the interval maximum function
\[
m_j:\overline{S}\to \uR
\]
is continuous (in the extended sense).
\end{lemma}
However, in line with our main object of study, we need to free ourselves from the singularity of the kernels here. Therefore, below we will work out a number of continuity type results dropping the singularity assumption on the kernel, but invoking various additional continuity-type assumptions on the $n$-field function $J$. First let us present\footnote{We think that this must have been observed and recorded by various authors in various forms. However, after not finding a reference we decided to present the few lines proof here.} an almost obvious, yet very useful partial substitute for the above continuity result, which is still surprisingly general, not requiring more assumptions on the field function $J$ than mere upper boundedness.

\begin{lemma}\label{lem:mjcontrestab}
	 Let $K_1,\dots, K_n:[-1,1]\to \uR$ be extended continuous functions and $J:[0,1]\to \uR$ be an arbitrary upper bounded function. Consider the corresponding generalized sum of translates function $F$ from \eqref{eq:Fsumgen}. Let $[a,b]\subseteq [0,1]$ be a closed interval. Then the function $m_{[a,b]}(\yy)=\sup_{t\in [a,b]}F(\yy,\cdot)$ is (extended) continuous on $\oS$.
	\end{lemma}
\begin{proof} For technical simplicity, let us consider $\Phi(\xx,t):=\exp(F(\xx,t))$ and
$$
\mu(\xx):=\exp(m_{[a,b]}(\xx))=\exp(\sup_{[a,b]} F(\xx,\cdot))=\sup_{[a,b]} \exp(F(\xx,\cdot))=\sup_{[a,b]} \Phi(\xx,\cdot).
$$
Also, put $f(\xx,t):=\sum_{j=1}^n K_j(t-x_j)
$ and $\ff(\xx,t):=\exp(f(\xx,t))$. As $K_j$ are all extended continuous, so is $f(\xx,t)$, whence $\ff(\xx,t)$ is standard (finite valued) continuous and is therefore uniformly continuous, too. Thus for any $\ve>0$ there is $\de>0$ such that in particular $|\ff(\xx,t)-\ff(\yy,t)|<\ve$ whenever $\|\xx-\yy\|<\de$.

Let now $\xx, \yy \in \oS$ be fixed with $\|\xx-\yy\|<\de$, and $t,s \in [a,b]$ be points such that $\Phi(\xx,t)>\mu(\xx)-\ve$ and $\Phi(\yy,s)>\mu(\yy)-\ve$. We then have
\begin{align*}
\mu(\yy) &\ge \Phi(\yy,t)=\exp(J(t)) \ff(\yy,t) \ge \exp(J(t)) (\ff(\xx,t)-\ve)
\\& = \Phi(\xx,t) -\exp(J(t)) \ve \ge \mu(\xx)-\ve(1+\exp(J(t)),
\end{align*}
and similarly $\mu(\xx) \ge \mu(\yy)-\ve (1+\exp(J(s))$. As $J$ is upper bounded, we also have $0\le \exp(J(\cdot)) \le C$ with a finite constant $C$, whence this means $|\mu(\xx)-\mu(\yy)|\le (C+1)\ve$ whenever $\|\xx-\yy\|<\de$. That is, $\mu$ is continuous on $\oS$, and therefore also $m_{[a,b]}=\log \mu$ is extended continuous.
\end{proof}

Two immediate consequences are the following:

\begin{lemma}\label{lem:mbarcont}
	Let $n\in \NN$, $\nu_j>0 ~(j=1,\ldots,n)$, let $J$ be an $n$-field function, and let $K$ be a kernel function. Then $\mol :\oS \to \RR$ is continuous.
\end{lemma}

\begin{lemma}\label{lem:mjcontrest}Let $K_1,\dots. K_n:[-1,1]\to \uR$ be extended continuous functions and $J:[0,1]\to \uR$ be an arbitrary upper bounded function.  Consider the corresponding generalized sum of translates function $F$ from \eqref{eq:Fsumgen}. For fixed $\xx\in \oS$ and $j\in\{0,1,\dots,n\}$ the function $m_j(\yy)=\sup_{t\in [y_j,y_{j+1}]}F(\yy,\cdot)$ is (extended) continuous on
\[
S_{j,\xx}:=\{\yy:y_j=x_j,\: y_{j+1}=x_{j+1}\}.
\]
\end{lemma}
\begin{proof}

Set $[a,b]:=[x_j,x_{j+1}]$. Then, with the notation of Lemma \ref{lem:mjcontrestab} $m_j|_{S_{j,\xx}}=m_{[a,b]}|_{S_{j,\xx}}$, and the assertion  follows.\end{proof}

\begin{lemma}\label{lem:mjcontrest2}Let $K_1,\dots. K_n:[-1,1]\to \uR$ be extended continuous functions and $J:[0,1]\to \uR$ be an arbitrary upper bounded function.  Consider the corresponding generalized sum of translates function $F$ from \eqref{eq:Fsumgen}. Fix $\xx \in \oS$ and $j \in \{0,1,\ldots,n\}$ and let $\calX_j:=\{k\in \{1,\dots,n\}: x_k=x_j\}$ and  take $i\in \{0,1,\dots,n\}$ with $x_j\not\in I_i(\xx)$.
Then the function $m_i(\yy)=\sup_{t\in [y_i,y_{i+1}]}F(\yy,\cdot)$ is (extended) continuous on \[
A_{i,j,\xx}:=\{\zz\in \oS: z_k=z_j\not\in I_i(\xx)\text{ for each $k\in \calX_j$ and $z_k=x_k$ for each $k\not\in \calX_j$}\}.
\]
\end{lemma}
\begin{proof}
Set $[a,b]:=[x_i,x_{i+1}]$. Then for every $\yy\in A_{i,j,\xx}$ one has $y_j\not\in I_i(\xx)=I_i(\yy)$.  The asserted continuity now follows from Lemma \ref{lem:mjcontrestab}.\end{proof}

For the following recall that in case $K$ is finite-valued on $[-1,1]$ (e.g. if $K$ is non-singular and monotone), then the pure sum of translates function is uniformly continuous on the compact set $\oS\times [0,1]$: for any $\varepsilon>0$ there is $\delta_0>0$ such that we have
\begin{align}
\label{puresumfuncontinuity}
	f(\vv,t)-\ve\leq f(\uu,s)&\leq f(\vv,t)+\ve\\
	\notag&\left(\forall s,t\in [0,1], |s-t|<\delta_0 \:\textrm{and} \: \forall \uu, \vv \in \oS,\|\uu-\vv\|<\delta_0  \right).
\end{align}

\begin{prop} \label{prop:mjlscusc}
Let $K$ be a kernel function, let $n\in \NN$, let $\nu_1,\dots, \nu_n>0$, let $J$ be an $n$-field function, and let $j\in\{0,1, \dots, n\}$. Moreover, in case $K$ is non-singular, assume also $K(-1), K(1)>-\infty$.
\begin{abc}
	\item If $J$ is upper semicontinuous, then  $m_j:\oS\to \uR$ is  upper semicontinuous.
	\item If $J$ is  lower semicontinuous, then $m_j:\oS\to \uR$ is  lower semicontinuous.
	\item If $J$ is continuous in the extended sense, then $m_j: \oS\to \uR$ is continuous in the extended sense.
		\end{abc}
\end{prop}
\begin{proof} We may assume that $K$ is non-singular, otherwise the continuity of $m_j$ is contained in Lemma \ref{lem:mjcont} without any further conditions on $J$. So, according to the finiteness assumption, $f:\oS\to \RR$ is a real-valued continuous function, hence uniformly continuous and for any $\varepsilon>0$ we have \eqref{puresumfuncontinuity}.

\medskip \noindent (a) 	Let $\xx\in\oS$ and let $L>m_j(\xx)$ be arbitrary. Take $\ve>0$ with $L-\ve>m_j(\xx)$ and the corresponding $\delta_0$ from uniform continuity. If $J(x_j)=-\infty$, let $a:=L-\ve-\sup f$, and if $J(x_j)>-\infty$, let $a:=J(x_j)+\ve$. Similarly, if $J(x_{j+1})=-\infty$, let $b:=L-\ve-\sup f$, and if $J(x_{j+1})>-\infty$, let $b:=J(x_{j+1})+\ve $.
By the upper semicontinuity of  $J$ there is  $\delta<\delta_0$ such that for every $t\in (x_j-\delta,x_j+\delta)\cap [0,1]$ we have $J(t)<a$, and for every $t\in (x_{j+1}-\delta,x_{j+1}+\delta)\cap [0,1]$  we have $J(t)<b$.

Let now $\yy\in \oS$ with $\|\xx-\yy\|<\delta$. First, take $t\in [y_j,y_{j+1}]\cap [x_j,x_{j+1}]$. Then
\[
F(\yy,t)=f(\yy,t)+J(t)\leq f(\xx,t)+J(t)+\ve\leq m_j(\xx)+\ve.
\]
Second, take $t\in [y_j,y_{j+1}]\setminus [x_j,x_{j+1}]$. Then $\min\{|t-x_j|,|t-x_{j+1}|\}<\delta$, so  either $|t-x_j|<\de$ or $|t-x_{j+1}|<\de$. We get in the first case $F(\yy,t)=f(\yy,t)+J(t)\leq f(\xx,x_j)+\ve+a$, and in the second that
$F(\yy,t)=f(\yy,t)+J(t)\leq f(\xx,x_{j+1})+\ve+b$.
Altogether we obtain
\begin{align*}
m_j(\yy)&=\sup_{[y_j,y_{j+1}]}F(\yy,\cdot)\leq \max\{m_j(\xx)+\ve,f(\xx,x_j)+\ve+a,f(\xx,x_{j+1})+\ve+b\} \leq  L
\end{align*}
for every $\yy\in \oS$ with $\|\yy-\xx\|<\delta$. This proves that $\limsup_{\yy\to \xx}m_j(\yy)\leq m_j(\xx)$.

\medskip \noindent (b) We are to prove lower semicontinuity, i.e. $\liminf_{\yy\to\xx} m_j(\yy)\ge m_j(\xx)$, whence if $m_j(\xx)=-\infty$, then there remains nothing to prove. So we can assume without loss of generality that $m_j(\xx)>-\infty$ is finite.

We distinguish two cases, the first being when $I_j(\xx)$ is degenerate, i.e. $x_j=x_{j+1}$, in which case we necessarily have $J(x_j)>-\infty$. Let $\ve>0$ be arbitrary, and $\de_0$ be the bound furnished by the uniform continuity of $f$, as written above. Also, let $\de\le \de_0$ be so small, that $J(t)>J(x_j)-\ve$ all over $(x_j-\de,x_j+\de)$. Such a $\de>0$ exists because $J$ is lower semicontinuous. Now if $\|\yy-\xx\|<\de$, then $y_j \in (x_j-\de,x_j+\de)$, and we are led to
\begin{align*}
m_j(\yy)&\geq F(\yy,y_j)=f(\yy,y_j)+J(y_j)
\\& \geq f(\xx,x_j)-\ve+J(x_j)-\ve=F(\xx,x_j)-2\ve=m_j(\xx)-2\ve,
\end{align*}
proving lower semicontinuity of $m_j$ in this case.

Let now $I_j(\xx)$ be non-degenerate, i.e. $x_j<x_{j+1}$. Take $z \in I_j(\xx)$ be a point with $F(\xx,z)>m_j(\xx)-\ve$. By lower semicontinuity of $J$ there is $\de_1>0$ such that for points $t\in (z-\de_1,z+\de_1)\cap [0,1]$ we have $J(t)>J(z)-\ve$.
We thus find some $w\in (z-\de_1,z+\de_1)\cap (x_j,x_{j+1})$ with $J(w)>J(z)-\ve$. Take $0<\de \le \min\{\de_0,\de_1,w-x_j,x_{j+1}-w\}$ and some $\yy$ with $\|\yy-\xx\|<\de$, then $w\in I_j(\yy)$ and hence
$m_j(\yy)\ge F(\yy,w)=f(\yy,w)+J(w) \ge f(\xx,z)-\ve + J(z)-\ve= F(\xx,z)-2\ve \ge m_j(\xx)-3\ve$. This proves lower semicontinuity also for this case.

\medskip \noindent (c) follows obviously from (a) and (b).
\end{proof}

We can strengthen the assertion (b) in the above lemma, more precisely we can relax the assumption of lower semicontinuity. Indeed, for $J$ lower semicontinuity is a stronger assumption than just postulating the below ``two-sided limsup condition'' for $J$.
\begin{equation}\label{eq:Jbothside}
\limsup_{s\upto t }J(s)\geq J(t)\quad\text{and}\quad\limsup_{s\downto t }J(s)\geq J(t) \qquad (\forall t \in [0,1])
\end{equation}
Of course, at the endpoints $t=0$ and $t=1$ only the one meaningful condition is required. Under this condition, however, we will get lower semicontinuity on a somewhat restricted subset of $S$.
\begin{prop} Let $n\in\NN, ~ \nu_j>0 ~(j=1,\ldots,n)$, and $K$ be a kernel function which is either singular or finite-valued. Further, assume the two-sided limsup condition \eqref{eq:Jbothside} holds for the $n$-field function $J$.

Then for each $j=0,\ldots,n$ the function $m_j$ is lower semicontinuous on the set $S_j:=\{\xx \in \oS~:~ x_j<x_{j+1}\}$.

In particular, all $m_j$ are lower semicontinuous on $S$.
\end{prop}
\begin{proof}
Similarly to the above, Lemma \ref{lem:mjcont} gives more than asserted in case $K$ is singular, so we can restrict ourselves to the other case, when $K$ is finite-valued hence $f$ is uniformly continuous.

As $\cap_{j=0}^n S_j=S$, the last assertion is clearly entailed by the first.

So let $j$, and also $\xx\in S_j$ be fixed. As before, we can assume $m_j(\xx)>-\infty$, for in case $m_j(\xx)=-\infty$ lower semicontinuity at $\xx$ is obvious. Take $\ve>0$ arbitrarily. Further, take a point $z \in I_j(\xx)$ where $F(\xx,z)>m_j(\xx)-\ve$.

If $z \in (x_j,x_{j+1})$, then consider a $\de_0>0$ corresponding to $\ve>0$ regarding the uniform continuity of $f$ on $\oS\times [0,1]$, and take $\de<\min(\de_0, |z-x_j|,|z-x_{j+1}|)$, see \eqref{puresumfuncontinuity}. Then for all $\yy$ with $\|\yy-\xx\|<\de$ we have $z \in I_j(\yy)$, and we get $m_j(\yy) \ge F(\yy,z)=f(\yy,z)+J(z)>f(\xx,z)-\ve+J(z)=F(\xx,z)-\ve\ge m_j(\xx)-2\ve$.

The somewhat more complicated part is when $z$ is at some endpoint, say $z=x_j$. However, in this case according to the two-sided limsup condition \eqref{eq:Jbothside} for the given $\de_0>0$ there must exist some $w\in (x_j,x_j+\de_0)$ satisfying $J(w)>J(x_j)-\ve$. It follows that with $\de< \min(\de_0, |w-x_j|,|w-x_{j+1}|)$ we will have $w \in I_j(\yy)$ for all $\yy$ with $\|\yy-\xx\|<\de$, and it holds $m_j(\yy)\ge F(\yy,w)= f(\yy,w)+J(w)>f(\xx,x_j)-\ve+J(x_j)-\ve =F(\xx,x_j)-2\ve=F(\xx,z)-2\ve>m_j(\xx)-3\ve$. The case $z=x_{j+1}$ is similar.

Altogether we infer that for any $\ve>0$ there is an appropriate $\de>0$ such that for all $\yy$ with $\|\yy-\xx\|<\de$ it holds $m_j(\yy)\ge m_j(\xx)-3\ve$, which means lower semicontinuity of $m_j$ at $\xx$.
\end{proof}
\begin{cor}\label{cor:Jconttwosided} Let $n\in \NN$, $\nu_j>0$ ($j=1,\ldots,n$) be given positive reals, let $K$ be an either singular or finite valued kernel function and let $J$ be an upper semicontinuous $n$-field function satisfying the two-sided limsup condition \eqref{eq:Jbothside}. Then $\mv$ is extended continuous on $S$.
\end{cor}

\begin{rem}
	 If the kernel is non-singular, one cannot drop the additional assumptions on $J$ from Proposition \ref{prop:mjlscusc} as compared to Lemma \ref{lem:mjcont}. Take e.g. $K\equiv0$ and $J=\log \car_{[1/4,1/2]}$ (which is upper semicontinuous but not lower semicontinuous and does not satisfy \eqref{eq:Jbothside}) and assertion (b) of Proposition \ref{prop:mjlscusc} fails, or $K\equiv0$ and $J=\log \car_{(1/4,1/2)}$ (which is not upper semicontinuous) and assertion (a) of Proposition \ref{prop:mjlscusc} fails.
	\end{rem}

In the following we will use yet another condition on the $n$-field function $J$, which relaxes \eqref{eq:Jbothside}.
This we will call ``the weak limsup condition''.
\begin{equation}\label{cond:weaklimsup}
\limsup_{s\to t }J(s)\geq J(t) \quad \text{for all $t \in [0,1]$.}
\end{equation}
This implies in particular that the finiteness domain $X^c$ has no isolated points, and together with the upper semicontinuity of $J$ yield the condition \eqref{cond:limsup} in Lemma \ref{lem:Nonmajorization} below. Our major example is $J:=\log \car_E$ with $E$ a finite union of non-degenerate closed intervals. The example can be generalized by allowing $E$ to be a perfect set, or invoking a positive (say, continuous) weight function $w$ on $E$ and taking $J:=\log (w\car_E)$.

\begin{lem} \label{lem:YS}
Let $n\in \NN$, $\nu_j>0$ ($j=1,\ldots,n$) be given positive reals, let $K$ be an either singular or finite valued kernel function and
suppose that the $n$-field function $J$ satisfies the weak limsup condition \eqref{cond:weaklimsup}.
Then for every $\xx\in Y\cap \partial S$  there is a sequence $(\xx_k)_{k\in \NN}$ in $Y\cap S$ such that $\xx_k\to \xx$ and
	\[
	\liminf_{k\to \infty}m_j(\xx_k)\geq m_j(\xx) \quad\text{for each $j\in \{0,1,\dots,n$}\}.
	\]
\end{lem}

\begin{proof}
If $K$ is singular, then $Y\cap \partial S=\emptyset$, hence the assertion is trivially true.

So let $K$ be non-singular and, by assumption, then also finite valued, so we have uniform continuity of the pure sum of translates function $f$ as is given in \eqref{puresumfuncontinuity}.
We shall prove that for every $\ve>0$ and every $\delta>0$
there is $\yy\in  Y\cap S$ with $\|\yy-\xx\|<\delta $ and
	\[
	 m_j(\yy)\geq m_j(\xx)-\ve \quad \text{for every}\quad  j\in \{0,1,\dots,n\}.
	 \]
Then this will establish also the assertion.

So let $\ve>0$ and $\delta>0$ be given arbitrarily and let $\delta_0\in (0,\delta)$
be chosen to $\ve$ from the uniform continuity of the pure sum of translates function $f$.

\medskip\noindent
As $\xx\in \partial S$, there are $x_\ell=\cdots=x_{r}$ with $\ell<r$ such that no node with index smaller than $\ell$ coincides with $x_\ell$ and no node with index greater than $r$ coincides with $x_r$.
It can happen that $\ell=0$ or $r=n+1$, but not both, so  we can suppose by symmetry that $r<n+1$ and $x_r<1$. Note that $\xx\in Y$ and $I_\ell(\xx)=\{x_r\}$ entails that $m_\ell(\xx)=F(\xx,x_r)>-\infty$, whence also $J(x_r)>-\infty$.

Now if $\ell=0$, then \eqref{cond:weaklimsup} means $\limsup_{t\downto x_r}J(t)\geq J(x_r)>-\infty$. If on the other hand $\ell>0$ and $0 \le x_{\ell-1}<x_\ell=x_r$, then either $\limsup_{t\upto x_r}J(t)\geq J(x_r)>-\infty$ or $\limsup_{t\downto x_r}J(t)\geq J(x_r)>-\infty$, but then by symmetry we can assume the latter again.

It follows that the set $A:=\{t:t\in (x_r,x_{r+1}),\:J(t)>J(x_r)-\ve\} $ has $x_r$ as an accumulation point (from the right), hence for any $\eta>0$ we can take $r-\ell+1$ different points $p_{\ell}<\cdots<p_r \in A\cap (x_r,x_r+\eta)$. We pick a point $z_r\in[x_r,x_{r+1}]$, too, such that $F(\xx,z_r)>m_r(\xx)-\ve$. If $z_r>x_r$ then we  take $\eta:=\min(\de_0,z_r-x_r)$ such that the points $p_\ell,\ldots,p_r$ lie even in $(x_r,z_r)$; otherwise we  just take $\eta:=\min(\de_0,x_{r+1}-x_r)$.

Now we replace in $\xx$, the node $x_{\ell+1}$ by $x'_{\ell+1}:=(p_{\ell}+p_{\ell+1})/2$, the node $x_{\ell+2}$ by $x'_{\ell+2}:=(p_{\ell+1}+p_{\ell+2})/2$, ..., and the node $x_r$ by $x_r':=(p_{r-1}+p_{r})/2$; we write $\xx'$ for the new node system. We then have $p_{i}\in (x'_{i},x'_{i+1})$  and $J(p_i)>J(x_r)-\ve $ for every $i=\ell,\dots, r$.
By construction $z_r \in I_r(\xx')=[x'_r,x_{r+1}]$, too, unless $z_r=x_r$.

\medskip\noindent  Let $j\in \{0,1,\dots,n\}$ be arbitrary. If $j<\ell$ or $j>r$, then we have for $t\in [x'_j,x'_{j+1}]=[x_j,x_{j+1}]$ that
\[
m_j(\xx')\geq F(\xx',t)=f(\xx',t)+J(t)\geq f(\xx,t)-\ve+J(t),
\]
(also using uniform continuity), so we obtain $m_j(\xx')\geq m_j(\xx)-\ve$.

\medskip\noindent If $\ell\leq j \le r$ we have  $p_j\in[x_j',x_{j+1}']\cap A$ and thus, with another reference to uniform continuity (see \eqref{puresumfuncontinuity}), conclude
\begin{align*}
m_j(\xx')&\geq F(\xx',p_j)=f(\xx',p_j)+J(p_j)\geq f(\xx,x_r)-\ve+J(p_j)\\
&\geq   f(\xx,x_r)-\ve+J(x_r)-\ve=F(\xx,x_r)-2\ve.
\end{align*}
Since for $\ell\le j<r$ it holds $m_j(\xx)=F(x_r)$, it follows that  $m_j(\xx')\ge m_j(\xx)-2\ve$, too. It remains to take care of the case $j=r$. In fact, if $j=r$ and $z_r=x_r$, too, then the above furnishes $m_r(\xx')\ge F(\xx,x_r)-2\ve=F(\xx,z_r)-2\ve \ge m_r(\xx)-3\ve$.

So let finally $j=r$ and $z_r>x_r$, hence also $z_r\in I_r(\xx')$. Then by uniform continuity of $f$ and $\|\xx-\xx'\|<\de_0$ we are led to
\begin{align*}
m_r(\xx')\geq f(\xx',z_r)+J(z_r)\geq f(\xx,z_r)-\ve +J(z_r)
=F(\xx,z_r)-\ve \geq m_r(\xx)-2\ve.
\end{align*}
In all, we find for each $j\in\{0,1,\ldots,n\}$ the inequality $m_j(\xx')\ge m_j(\xx)-3\ve$.

Repeating this argument for all $k<n$  distinct maximal groups of coinciding nodes we therefore get $m_j(\xx^{(k)})\ge m_j(\xx)-3k\ve$ and $\xx^{(k)}\in  Y\cap S$, which finishes the proof.
\end{proof}

By analyzing the previous proof one can also see the following.
\begin{lem} \label{lem:YSdense}Let $n\in \NN$, $\nu_j>0$ ($j=1,\ldots,n$) be given positive reals, let $K$ be an either singular or finite valued kernel function and suppose that the finiteness domain $X^c$ of the $n$-field function $J$ has no isolated points.
    Then $Y \cap S$ is dense in $Y$.
\end{lem}

\begin{proof} We only need to prove that any $\xx\in Y\cap \partial S$ can be arbitrarily well approximated from $Y\cap S$. So fix $\xx\in Y\cap\partial S$.
 Then there are $x_\ell=\cdots=x_{r}$ with $\ell<r$ such that no node with index smaller than $\ell$ coincides with $x_\ell$ and no node with index greater than $r$ coincides with $x_r$. It can happen that $\ell=0$ or $r=n+1$, but not both, so  we can suppose by symmetry that $r<n+1$ and $x_r<1$. Note that $\xx\in Y$ and $I_\ell(\xx)=\{x_r\}$ entails that $m_\ell(\xx)=F(\xx,x_r)>-\infty$, hence also $J(x_r)>-\infty$. Take $\delta>0$ arbitrarily.

\medskip\noindent  The set $[0,1]\setminus X=X^c$  has $x_r$ as an accumulation point,
so, e.g., the  set
$A:=\{t:t\in (x_r,x_{r+1}), J(t)>-\infty\} $ has $x_r$ as an accumulation point, too. Hence for any $\eta$ with $0<\eta<\min\{\delta,x_{r+1}-x_r\}$ we can take $r-\ell+1$ different points $p_{\ell}<\cdots<p_r \in A\cap (x_r,x_r+\eta)$. Now we replace in $\xx$, the node $x_{\ell+1}$ by $x'_{\ell+1}:=(p_{\ell}+p_{\ell+1})/2$, the node $x_{\ell+2}$ by $x'_{\ell+2}:=(p_{\ell+1}+p_{\ell+2})/2$,
...,
and the node $x_r$ by $x_r':=(p_{r-1}+p_{r})/2$; we write $\xx'$ for the new node system. We then have $p_{i}\in (x'_{i},x'_{i+1})$  and $J(p_i)>-\infty $ for every $i=\ell,\dots, r$. It follows that $m_i(\xx')>-\infty$  for every $i=\ell,\dots, r$, while $\|\xx-\xx'\|<\delta$. Note also that $I_j(\xx)=I_j(\xx')$ for  $j<i$ and $j>r$, hence $m_j(\xx')>-\infty$ in view of $m_j(\xx)>-\infty$, and we find $\xx'\in Y$, too. 

Repeating this argument for all $k<n$ distinct maximal groups of coinciding nodes we therefore obtain $\xx^{(k)}\in Y\cap S$ with $\|\xx-\xx^{(k)}\|< \delta$.
\end{proof}

\begin{lem} \label{lem:WdenseinY}
Let $n\in \NN$, $\nu_j>0$ ($j=1,\ldots,n$) be given positive reals, let $K$ be an either singular or finite valued kernel function and suppose that the finiteness domain $X^c$ of the $n$-field function $J$ has no isolated points. Then the set
\begin{equation}\label{eq:Wdef}
	W:=\{\ww:\ww\in Y\cap S,\: \rint [w_j,w_{j+1}]\cap X^c\neq\emptyset\text{ for } j=0,1,\dots,n\}
\end{equation}
is dense in $Y$, where $\rint$ denotes the relative interior within $[0,1]$.
\end{lem}
\begin{proof}
According to Lemma \ref{lem:YSdense},
$Y\cap S$ is dense in $Y$. We need to prove  that for all $\xx\in Y \cap S$ either $\xx \in W$, or for any $\ve>0$ there are points of $W$ closer to $\xx$ than $\ve$.

So let $\ve>0$ and $\xx \in Y \cap S$ be fixed arbitrarily. Denote by $k\le n+1$ the number of  indices $j\in \{0,1,\dots,n\}$ with $\rint [x_j,x_{j+1}]\cap X^c =\emptyset$ (call such indices ``exceptional''). We prove the assertion by induction on $k$. If $k=0$ then $\xx \in W$, and then there remains nothing to prove. So assume $k>0$ and consider the statement holding for all smaller values of the number of exceptional indices.

As $k>0$, there is an index $j$ with $\rint [x_j,x_{j+1}]\cap X^c =\emptyset$.
This implies either $J(x_j)>-\infty$ or $J(x_{j+1})>-\infty$ (or both), because otherwise we would have $F|_{I_j(\xx)}\equiv -\infty$, hence $m_j(\xx)=-\infty$, contradicting to $\xx \in Y$. Suppose that, e.g., $J(x_j)>-\infty$ holds (the other case being similar). Note that in this case we must have $j>0$ and $x_j>0$, because for $j=0$ and $\xx\in S$, we have $x_0=0\in \rint I_j(\xx)$, and $\rint I_j(\xx)\cap X^c=\emptyset$ prohibits $x_0\in X^c$, i.e., we must have $J(x_0)=-\infty$. 

Let $\eta:=\min(\ve,x_{j+1}-x_j,x_j-x_{j-1})$.
Since $\xx\in Y\cap S$, we have $\eta>0$.
According to the assumption, $x_j$ is a limit point of $X^c$. As $F$ is identically $-\infty$ on $(x_j,x_j+\eta)\subset \rint I_j(\xx)$, we therefore
have some points $z \in (x_j-\eta,x_j) \subset (x_{j-1},x_j)$ satisfying $J(z) >-\infty$.

Now we define $\yy$ to be the node system obtained from $\xx$ by changing $x_j$ to some point $y_j\in (z,x_j)$ (which is possible by construction), and retaining all the other nodes putting $y_i:=x_i~ (i\ne j)$. Then $\|\yy-\xx\|=|y_j-x_j|<x_j-z<\eta\le \ve$. Also, $y_{j-1}:=x_{j-1}\le x_j-\eta <z < y_j<x_j<y_{j+1}:=x_{j+1}$, so  together with $\xx$ also $\yy \in S$. It is easy to see that we also have $\yy \in Y$, because for the only two intervals changed $m_{j-1}(\yy)\ge F(\yy,z)=f(\yy,z)+J(z)>-\infty$ and $m_j(\yy)\ge F(\yy,x_j)=f(\yy,x_j)+J(x_j)>-\infty$, the latter holding because $x_j \in (0,1) \setminus \{y_1,\ldots,y_n\}$, whence $f(\yy,x_j)>-\infty$, while $J(x_j)>-\infty$ by construction.

Note also that $z \in \rint I_{j-1}(\yy)$ and $x_j\in \rint I_j(\yy)$, while $\rint I_i(\xx)=\rint I_i(\yy)$ for all the other indices $i$, so  the number of exceptional indices for $\yy$ is less than that for $\xx$.

By the inductive hypothesis we thus have some $\ww \in W$ such that for the node system $\yy \in Y\cap S$ the distance is $\|\ww-\yy\|< \ve$. The assertion is proved.
\end{proof}

\begin{lem}\label{lem:W}
Let $n\in \NN$, $\nu_j>0$ ($j=1,\ldots,n$) be given positive reals, let $K$ be an either singular or finite valued kernel function and suppose that the $n$-field function $J$ satisfies \eqref{cond:weaklimsup} (so its finiteness domain $X^c$ has no isolated points).
	
Then for every $\xx\in Y\cap S$ there is a sequence $(\xx_k)_{k\in \NN}$ in $W$, where $W$ is defined in \eqref{eq:Wdef}, such that
$\xx_k\to \xx $ and
\[
\liminf_{k\to \infty}m_j(\xx_k)\geq m_j(\xx) \quad\text{for each $j\in \{0,1,\dots,n$}\}.
\]
\end{lem}
\begin{proof}
We can assume that $K$ is non-singular, since otherwise the functions $m_j$ are continuous according to Lemma \ref{lem:mjcont}, and then any sequence  $(\xx_k)_{k\in \NN}$ in $W$ converging to $\xx$ will do (and there exists some such sequence by Lemma \ref{lem:WdenseinY}).

We shall prove that for every $\ve, \delta>0$
there is $\yy\in W$ with $\|\yy-\xx\|<\delta$ and
	 \[
	 m_j(\yy)\geq m_j(\xx)-\ve\quad\text{for every $j\in \{0,1,\dots,n\}$,}
	 \]
which will clearly imply the statement of the lemma. Let $\delta_0\in (0,\delta)$ be chosen  to $\ve$ from the uniform continuity of the pure sum of translates function $f$ (see \eqref{puresumfuncontinuity}).

Let $j\in \{0,1,\dots,n\}$, and suppose  $\rint I_j(\xx)\cap X^c=\emptyset$. Since $\xx\in Y$, this implies either $J(x_j)>-\infty$ or $J(x_{j+1})>-\infty$ (or both). Suppose that, e.g., $J(x_j)>-\infty$ holds (the other case being similar). Note that in this case we must have $j>0$ and $x_j>0$, because $\xx\in S$, $J(x_j)>-\infty$ and  $\rint I_j(\xx)\cap X^c=\emptyset$.
Since  $\limsup_{t\to x_j}J(t)\geq J(x_j)$,---by the weak limsup condition \eqref{cond:weaklimsup}---we must have  $\limsup_{t\upto x_j}J(t)\geq J(x_j)>-\infty$, hence also $\limsup_{t\to x_j}F(\xx,t)\geq F(\xx,x_j)$. Moreover, there exists  $z_{j-1}\in[x_{j-1},x_j)\cap (x_j-\delta_0,x_j)$ with  $F(\xx,z_{j-1})\geq m_{j-1}(\xx)-\ve$.
Let $x'_j\in (z_{j-1},x_j) \subset (x_j-\delta_0,x_j)$, and
replace in $\xx$ the node $x_j$ by $x'_j$, the new node system being denoted by $\xx'$. We obviously have $x_j \in \rint I_j(\xx')\cap X^c$ hence in particular $\rint I_j(\xx') \cap X^c \neq \emptyset$.

\medskip\noindent Let $i\in \{0,1,\dots,n\}$. If $i<j-1$ or $i\geq j$ we have for $t\in [x_i,x_{i+1}]\subseteq [x'_i,x'_{i+1}]$ that
\[
m_i(\xx')\geq f(\xx',t)+J(t)\geq f(\xx,t)-\ve+J(t),
\]
so we obtain $m_i(\xx')\geq m_i(\xx)-\ve$.
	
\medskip\noindent  For $i=j-1$, we have  $z_{j-1}\in I_{j-1}(\xx')$ hence
\begin{align}\label{eq:ijminusone}
m_i(\xx')=m_{j-1}(\xx') & \geq f(\xx',z_{j-1})+J(z_{j-1})\geq f(\xx,z_{j-1})-\ve+J(z_{j-1})\\
\notag&=F(\xx,z_{j-1})-\ve\geq m_{j-1}(\xx)-2\ve.
\end{align}
Repeating this argument for all indices $j$, for which $\rint I_j(\xx)\cap X^c=\emptyset$, we arrive at some $\xx^* \in W$ with $\|\ww-\xx^*\|<n\de$ and $m_i(\xx^*)>m_i(\xx)-2 n\ve$ ($i=0,1,\ldots,n$) concluding the proof.
\end{proof}

\begin{cor}\label{cor:seqcont}
Let $n\in \NN$, $\nu_j>0$ ($j=1,\ldots,n$) be given positive reals, let $K$ be an either singular or finite valued kernel function and suppose $J$ is an upper semicontinuous $n$-field function and, moreover, satisfies
\eqref{cond:weaklimsup} (thus  its finiteness domain $X^c$ has no isolated points).

Then $J$ also satisfies the so-called ``limsup condition''
\begin{equation}\label{cond:limsup}
\limsup_{s\to t} J(s)=J(t)\quad \text{for every $t\in[0,1]$,}
\end{equation}
and we have that for every $\xx\in Y$ there is a sequence $(\xx_k)_{k\in \NN}$ in $W$ such that $\xx_k\to \xx $ and
\[
\lim_{k\to \infty}m_j(\xx_k)=m_j(\xx)\quad\text{for each $j\in \{0,1,\dots,n\}$}.
\]
\end{cor}
\begin{proof}
Upper semicontinuity means $\limsup_{s\to t} J(s) \le J(t)$, while condition \eqref{cond:weaklimsup} is the reverse inequality, hence they together give \eqref{cond:limsup}.

Let $\xx \in Y$. According to Lemma \ref{lem:W} there is a sequence $(\xx_k)$ in $Y\cap S$ such that $\|\xx_k-\xx\|<\frac 1k$ and  $m_j(\xx_k)\ge m_j(\xx)-\frac1k$ for all $k\in \NN$ and for all $j\in \{0,1,\dots,n\}$. According to Lemma \ref{lem:YS} there are sequences $(\xx_{k,\ell})$  in $W$ such that $\|\xx_{k,\ell}-\xx_k\|<\frac 1\ell$ and $ m_j(\xx_{k,\ell})\ge m_j(\xx_k)-\frac1\ell$  for all $k,\ell\in \NN$ and for all $j\in \{0,1,\dots,n\}$. Consider the diagonal sequence $(\xx_{k,k})$ for which we have $\|\xx_{k,k}-\xx\|\leq \|\xx_{k,k}-\xx_k\|+\|\xx_{k}-\xx\|< \frac 2k$ and
\[
m_j(\xx_{k,k})\geq  m_j(\xx_k)-\frac1k\geq  m_j(\xx)-\frac2k
\]
for every  $k\in \NN$ and for all $j\in \{0,1,\dots,n\}$.
It follows that $\liminf_{k\to\infty} m_j(\xx_{k,k}) \ge m_j(\xx)$, while $\limsup_{k\to\infty} m_j(\xx_{k,k}) \le m_j(\xx) $ holds according to Proposition \ref{prop:mjlscusc} (a). Therefore, $\lim_{k\to\infty} m_j(\xx_{k,k}) = m_j(\xx)$ .
\end{proof}

\section{A result about the failure of strict majorization}
\label{sec:interproof}

\begin{lem}\label{lem:Nonmajorization} Let $n\in \NN$, $\nu_j>0$ ($j=1,\ldots,n$) be given positive reals, let $K$ be a monotone \eqref{cond:monotone} kernel function, and let $J$ be an $n$-field function satisfying the ``limsup condition'' \eqref{cond:limsup} (in particular, $J$ is upper semicontinuous).

Then \emph{strict majorization} $m_i(\xx)>m_i(\yy)$ for every $i=0,1,\ldots,n$ cannot hold between any two node systems $\xx,\yy\in Y$.
\end{lem}

\begin{proof}
We can assume that $K$ is real valued, i.e., non-singular, otherwise the statement is contained in Theorem  \ref{thm:mainold}. (Recall that once we have \eqref{cond:monotone}, non-singularity entails finiteness of $K$.)

We argue by contradiction and suppose that for some $\xx,\yy\in Y$ one has
\begin{equation}\label{eq:mjxy}
m_j(\xx)>m_j(\yy)\quad \text{for each $j=0,\dots,n$.}
\end{equation}
According to Corollary \ref{cor:seqcont} we can assume that $\xx, \yy \in W\subseteq S$.

Since $K$ is real valued the pure some of translates function $f$ is continuous, hence uniformly continuous, on the compact set $\ol{S}\times[0,1]$. Therefore, for each  $\ve>0$, in particular for $\ve:=\frac12\min\{m_j(\xx)-m_j(\yy):j=0,1,\dots,n\}$,
there is $\delta>0$ such that \eqref{puresumfuncontinuity} holds.

Since $m_j$ is a supremum, there exist points $t_j \in I_j(\xx)$  with $F(\xx,t_j)>m_j(\xx)-\ve$ ($j=0,1,\ldots,n$). This implies that $J(t_j)>-\infty$, because $m_j(\xx)>-\infty$.  Let $\T:=\{t_0,t_1,\ldots,t_n\}$.

Note that it can happen that $\#\T <n+1$, in case two points $t_i$ coincide; that is the case when $t_i=x_{i+1}=t_{i+1}$. However, only this type of coincidence can happen, as $t_i\in I_i(\xx)$, and $\xx\in S$ excludes intersection of more than two of the base intervals $I_i(\xx)$.

Let us separate two subsets of $\T$: the ``interior part'' $\I:=\{t_i \in\T~:~ t_i\in \rint I_i(\xx)\}$ and the ``boundary part'' $\B:=\{t_i\in \T ~:~ t_i\in (I_i(\xx)\setminus\rint I_i(\xx))\}=\T \cap \{x_1,\ldots,,x_n\}$. Let $\alpha_0:=\min\{ |t-x_k|: t \in \I, k=1,\ldots,n\}$ if $\I\neq \emptyset$, otherwise set $\alpha_0:=1$. Since for $t\in\I$ we have $t=t_i\in \rint I_i(\xx)$ for some $i\in\{0,1,\ldots,n\}$, this quantity is positive.

Take $\alpha\in (0,\alpha_0)\cap (0,\delta)$  arbitrarily. In view of the assumption \eqref{cond:limsup}, for each $t\in \B$ there is a  $t^*\in (t-\al,t+\al)\cap [0,1]$, distinct from $t$, such that $J(t^*) > J(t)-\ve$, too (and in fact there are infinitely many such points). Note that $t\in \B$ meant that we had (with one or two indices $i$) $t=t_i=x_j$ for some uniquely determined node $x_j$ with $1\le j\le n$ (all $x_j$ are different for $\xx \in S$). For this index $j$, corresponding to the selected $t$, let us write $w_{j-1}:=\min(t,t^*)$ and $z_j:=\max(t,t^*)$.

Now we define a new node system $\xx'$ with the following simple rule. If for some $j$ the node $x_j$ is matching with some of the points $t\in \T$ (i.e. $t=t_i=x_j\in \B$), then we replace $x_j$ by $x_j':=\frac12(w_{j-1}+z_j)$. If $x_j\not\in \B=\T\cap \{x_1,\ldots,x_n\}$, then we just leave $x_j':=x_j$.
Note that if at all $x_j'\ne x_j$, then we have $|x_j'-x_j|=|x_j-\frac{w_{j-1}+z_j}{2}|=|x_j-\frac{t+t^*}{2}|=\frac{|t^*-x_j|}{2} <\al/2$. This means in particular that if $\al<\min \{x_{j+1}-x_j:j=0,\ldots,n\}$, too (which minimum is positive because $\xx\in S$), then we have $0<x_1'<\cdots<x_n'<1$, hence the ordering of the nodes remained intact, and $\xx'\in S$, too.

There are two important properties of the new node system $\xx'$. The first one has already been observed: $\|\xx'-\xx\|< \al/2\leq \de$. Therefore we also have $|f(\xx',s')-f(\xx,s)|<\ve$ for all $s,s'\in [0,1]$ with $|s'-s|<\de$. It follows that in particular $|F(\xx',w_{j-1})-F(\xx,x_j)|, |F(\xx',z_j)-F(\xx,x_j)| <\ve$ and $F(\xx',w_{j-1}),F(\xx',z_j)>F(\xx,x_j)-\ve$. Second, we claim that for each $i\in\{0,1,\dots,n\}$ there is $s_i\in\rint I_i(\xx')$ satisfying
\begin{equation}\label{eq:miF}
F(\xx',s_i)>m_i(\xx)-2\ve.
\end{equation}
To see this, consider first the case when $t_i\in \I$. Then $t_i\in \rint I_i(\xx)$, and by construction its distance from $\{x_1,\ldots,x_n\}$ is at least $\alpha_0>\al$, hence even for $\xx'$ we have $s_i:=t_i\in \rint I_i(\xx')$. By choice of the $t_i$ and uniform continuity of $f$, however, we must have $F(\xx',s_i)=F(\xx',t_i)>F(\xx,t_i)-\ve>m_i(\xx)-2\ve$, as needed. Second, consider the case when $t_i\in \B$, so  $t_i=x_j$ for some $j$. By construction $w_{j-1}<x_j'<z_j$. So if $i=j-1$ then $x'_{j-1}\le x_{j-1}+\al/2<x_j-\al/2<w_{j-1}<x_j'$ and $w_{j-1} \in \rint I_{j-1}(\xx')$. So if $i=j-1$ we can choose $s_i=w_{j-1}\in \rint I_{j-1}(\xx')$ and have $F(\xx',s_i)=F(\xx',w_{j-1})>F(\xx,x_j)-\ve>m_i(\xx)-2\ve$, for in case $i=j-1$ and $x_j=t_{j-1}$ we had $F(\xx,x_j)=F(\xx,t_{j-1})>m_{j-1}(\xx)-\ve$. Similarly, if $i=j$, then we choose $s_j=z_j$, noting that $x_j'<z_j<x_j+\al/2<x_{j+1}-\al/2<x_{j+1}'$ and $z_j\in\rint I_j(\xx')$. Then we have with $s_j:=z_j$ that $F(\xx',s_i)=F(\xx',z_{j})>F(\xx,x_j)-\ve>m_i(\xx)-2\ve$, for in case $i=j$ and $x_j=t_j$ we had $F(\xx,x_j)=F(\xx,t_j)>m_j(\xx)-\ve$.
This establishes the asserted, second property of $\xx'$.

In all, from \eqref{eq:miF} we get for each $i=0,1,\ldots,n$ that $m_i(\xx')\ge F(\xx',s_i)>m_i(\xx)-2\ve$. It follows that, in particular, $\xx'\in W$, too.

\medskip\noindent
Finally, set $\eta_0:=\min\{|x'_j-s_i|:i,j=1,\dots,n\}>0$, and consider the singular kernel functions $K^{(\eta)}(t):=K(t)+\log_{-}(|t|/\eta)$ where $\log_-=\min (\log,0)$.  Denote by $F^{(\eta)}$ and $m_j^{(\eta)}$ the sum of translates function and the interval maxima functions for the kernel $K^{(\eta)}$ and the same $n$-field function $J$.
Since $K^{(\eta)}\leq K$, we immediately conclude $m_j^{(\eta)}\leq m_j$ for each $j\in \{0,1,\dots,n\}$. Note that for $\eta<\eta_0$ we have
$F^{(\eta)}(\xx',s_j)=F(\xx',s_j)$, because all the $\log_{-}$ terms vanish. Therefore, for $\eta<\eta_0$ and $j\in \{0,1,\dots,n\}$ we conclude $m_j^{(\eta)}(\xx') \ge F(\xx',s_j)> m_j(\xx)-2\ve>m_j(\yy)\ge m_j^{(\eta)}(\yy)$, so  $m_j^{(\eta)}(\xx')> m_j^{(\eta)}(\yy)$. Now, since $K^{(\eta)}$ is singular and, by $\yy\in W$, each $m_j^{(\eta)}(\yy)$, $j=0,1,\dots,n$ is finite---hence $\yy$ is in the regularity set $Y^{(\eta)}$ belonging to the perturbed kernel and the original field function---a contradiction with Theorem  \ref{thm:mainold} ensues.
\end{proof}

\section{The case of upper semicontinuous field functions}

\begin{lem}\label{lem:Japproxlem}
Let $g_k:[a,b]\to \uR$ be upper semicontinuous functions with $g_{k+1}\le g_k$ for each $k\in \NN$.
Let $g(x)=\lim_{k\to \infty} g_k(x)$.
Then we have  $\lim_{k\to\infty}\max g_k=\max g$.
\end{lem}
\begin{proof}
By the upper semicontinuity the maxima in question exist,
we set $\gamma:=\max g$ and $\gamma_k:=\max g_k$, and observe that $\gamma^*:=\lim_{k\to \infty} \gamma_k$ exists and satisfies $\gamma^*\geq \gamma$.

Take $t_k\in[a,b]$ with $g_k(t_k)=\gamma_k$.
By compactness we may pass to a subsequence, and
assume right away that $t_k\to t^*$ for some  $t^*\in[a,b]$.
Let $\alpha>\gamma\geq g(t^*)$.
Since $g_k(t^*)\to g(t^*)$ there is $k_0\in \NN$ such that $g_{k_0}(t^*)<\alpha$.
Since $g_{k_0}$ is upper semicontinuous,
there is $\delta>0$ such that for each $s\in[t^*-\delta, t^*+\delta]\cap [a,b]$ we have $g_{k_0}(s)<  \alpha$.
Since $t_n\to t^* $,
there is $n_0\in \NN$ such that $t_n\in[t^*-\delta, t^*+\delta] $ for every $n\geq n_0$.
Now if $n\geq \max\{k_0,n_0\}$, then $t_n\in [t^*-\delta, t^*+\delta]$ and
\[
\gamma_n=g_n(t_n)\leq g_{k_{0}}(t_n)<\alpha.
\]
This yields $\gamma^*\leq \gamma$, proving the statement.
\end{proof}
		
\blue{ITT TARTUNK}	
\begin{prop}[No strict majorization]\label{prop:nonmajorize}
Let $n\in \NN$, $\nu_1,\dots,\nu_n>0$, let $K$ be a monotone  \eqref{cond:monotone}  kernel function and let $J$ be an upper semicontinuous $n$-field function. Then there are no $\xx,\yy\in Y$ with $m_j(\xx)>m_j(\yy)$ for each $j\in \{0,1,\dots,n\}$.
\end{prop}
	
 \begin{proof}
We may assume that $K$ is non-singular, otherwise the statement is already contained in Theorem \ref{thm:mainold}.
Let $J^{(k)}$ be a sequence of continuous $n$-field functions such that $J^{(k)}\downto J$ pointwise.
Denoting the corresponding interval maxima by $m_j^{(k)}$, Lemma \ref{lem:Japproxlem} yields $m_j^{(k)}\downto m_j$ pointwise on $\oS$.
Suppose that for some $\xx,\yy\in Y$  we have $m_j(\xx)>m_j(\yy)$ for each $j\in \{0,1,\dots,n\}$. Then for large enough $k$ we have
$m^{(k)}_j(\xx)\geq m_j(\xx)>m^{(k)}_j(\yy)\geq m_j(\yy)>-\infty$  for every $j\in \{0,1,\dots,n\}$
(in particular, $\xx,\yy\in Y^{(k)}$, where $Y^{(k)}$ denotes the regularity set for the case of the  field function $J^{(k)}$, which is continuous, hence satisfies \eqref{cond:limsup}). This is, however, an impossibility in view of Lemma \ref{lem:Nonmajorization}.
	\end{proof}

\begin{prop}\label{prop:eqexistsmm}
Let $n\in \NN$, $\nu_1,\dots,\nu_n>0$, let $K$ be a monotone \eqref{cond:monotone}  kernel function and let $J$ be an upper semicontinuous $n$-field function.
Then there is an equioscillation point $\ee\in Y$ with
\[
\mol(\ee)=\inf_{\xx\in \oS}\mol(\xx)=M(\oSS),
\]
\end{prop}
\begin{proof}
We may assume that $K$ is non-singular,
otherwise the statement is already known,
see Theorem  \ref{thm:mainold}.

First we shall assume that $K$ is strictly concave and hence by \eqref{cond:monotone} also strictly monotone.
Since $\mol$ is continuous on $\oS$ (see Lemma \ref{lem:mbarcont}),
it attains its
infimum on $\oS$.
Take any such minimum point $\ee$.
We claim that $\ee$ is an equioscillation point. Assume the contrary, i.e.,
that there is $j\in \{0,1,\dots,n\}$ with $m_j(\ee)<\mol(\ee)$.
Let us take the smallest such $j$.

\medskip
If $j=0$, then $e_1<1$
otherwise $I_0(\ee)=[0,1]$
and $I_1(\ee)=\ldots=I_n(\ee)=\{1\}$,
hence $m_0(\ee)\ge m_1(\ee)=\ldots =m_n(\ee)$.
So we can move the node $e_1$ (and all others coinciding with it) slightly to the right.
By strict monotonicity and by Lemma \ref{lem:trivi} for the new node system $\ee'$ the values of $m_i(\ee')$, $i>0$  will strictly decrease provided $m_i(\ee)>-\infty$,
i.e.,  $\max_{i>0}m_i(\ee')<\max_{i>0}m_i(\ee)=\mol(\ee)$.
By the upper semicontinuity of $m_0$ (see  Proposition \ref{prop:mjlscusc} (a))
we can achieve that $m_0(\ee')<\mol(\ee)$, too. Altogether we obtain $\mol(\ee')<\mol(\ee)$,
a contradiction.

\medskip
The case $j=n$ can be handled analogously,
because in this case $e_n>0$, and
we can move $e_{n}$ to the left.

\medskip
Next we consider the case when $0<j<n$.
By the choice of $j$ we have $e_j>0$
(because, if $e_j=0$, then
$I_0(\ee)=\ldots=I_{j-1}(\ee)=\{0\}$,
so
$m_0(\ee)=m_j(\ee)<\mol(\ee)$ and $j$ was chosen minimal).
If $e_{j+1}=1$,
then $\mol(\ee)>m_j(\ee)\geq m_{j+1}(\ee)=\cdots =m_n(\ee)$,
and we can move $e_j$  (and all nodes coinciding with it) slightly to the left such that for
the new node system $\ee'$ one has $\mol(\ee')<\mol(\ee)$.
This is again a contradiction.

\medskip
So it remains the case when $0<j<n$, and $0<e_j\leq e_{j+1}<1$.
In this case we move $e_j$ slightly to the left,
$e_{j+1}$ slightly to the right (and all the other nodes coinciding with them), in such a way that with the total weights $p:=\sum_{i ~:~e_i=e_j} \nu_i$ and $q:=\sum_{i~:e_i=e_{j+1}} \nu_i$ and with $\alpha:=e_j', a:=e_j, b:=e_{j+1}, \beta:=e_{j+1}'$ we will have $\kappa=1$ for $\kappa$ in \eqref{eq:mudef}. Then we arrive at a new node system $\ee'$ with $m_j(\ee')$ still below $\mol(\ee)$, if the move is small enough, in view of $m_j(\ee)<\mol(\ee)$ and $m_j$ being upper semicontinuous.

For the other indices $i\ne j$, however, for the respective intervals $I_i(\ee')\subset I_i(\ee)$, moreover, by Lemma \ref{lem:widening} (c) and (d) also the corresponding values of the pure sum of translates function $f(\ee',\cdot)$ decreased strictly, so  by Lemma \ref{lem:trivi} also $m_i(\ee')<m_i(\ee)$ whenever $m_i(\ee)>-\infty$, hence in particular $m_i(\ee')<\mol(\ee)$. This holding both for $i=j$ and also for all $i\ne j$, finally we obtain $\mol(\ee')<\mol(\ee)$, contradicting to minimality of $\mol(\ee)$.
So finally we see that $\ee$ must be an equioscillation point.
The assertion is thus proved for the case of strictly concave monotone kernel functions.	
	
\medskip
Let $K$ be a general non-singular monotone  \eqref{cond:monotone}  kernel function
and consider for $\eta>0$ the strictly concave and monotone kernel functions
$K^{(\eta)}(t)=K(t)+\eta\sqrt{|t|}$.
We also introduce $m_j^{(\eta)}$,
$\mol^{(\eta)}$,
and $M^{(\eta)}(\oSS)$, for the corresponding quantities in the case of the kernel function $K^{(\eta)}$.
Since for $\eta\downto 0$ we have $K^{(\eta)}\downto K$,
we have $m_j^{(\eta)}\downto m_j$ and $\mol^{(\eta)}\downto \mol$ pointwise,
and since $\mol:\oS\to \RR$ is continuous
(see Proposition \ref{lem:mbarcont}), $\mol^{(\eta)}\downto \mol$ even uniformly.
It follows that $M^{(\eta)}(\oSS)\downto M(\oSS)$.
Let $\ee^{(\eta)}\in \oS$
be such that
$\mol^{(\eta)}(\ee^{(\eta)})
=M^{(\eta)}(\oSS)$,
so $\ee^{(\eta)}$ is an equioscillation point
by the already established first part.
Letting $\eta\downto 0$ we find a sequence $\eta_k\downto 0$
and $\ee \in \oS$
such that $\ee^{(\eta_k)}\to \ee$
as $k\to \infty$.

We claim that $\mol(\ee)=M(\oSS)$, the inequality $M(\oSS)\leq \mol(\ee)$ being trivial. Let $a>M(\oSS)$. Then for all sufficiently large $k$ we have $a\geq M^{(\eta_k)}(\oSS)$, so we can conclude
\[
a\geq M^{(\eta_k)}(\oSS)=\mol^{(\eta_k)}(\ee^{(\eta_k)})\geq \mol(\ee^{(\eta_k)}).
\]
Letting $k\to\infty$ we conclude,
by the continuity of $\mol$ and by $\ee^{(\eta_k)}\to \ee$,
that
$a\geq \mol(\ee)$ and $M(\oSS)\geq \mol(\ee)$ follows.
The claim is proved.

\medskip\noindent
Next we claim that $\ee$ is an equioscillation point.
Suppose the contrary, i.e., that for some $j\in \{0,1,\dots,n\}$
we have $m_j(\ee)<\mol(\ee)$.
Then there is $k_0\in \NN$ such that $m_j^{(\eta_{k_0})}(\ee)<\mol(\ee)$.
Since $m^{(\eta_{k_0})}_j$ is upper semicontinuous,
see Proposition \ref{prop:mjlscusc},
there is $\delta>0$ such that for every $\xx\in \oS $
with $\|\xx-\ee\|<\delta$
one has
$m_j^{(\eta_{k_0})}(\xx)
<\mol(\ee)$, too.

There is $n_0\in \NN$ such that for every $k\geq n_0$
we have $\|\ee^{(\eta_{k})}-\ee\|<\delta$.
So for $k\geq\max\{k_0,n_0\}$ we can write (since $\mol(\ee)=M(\oSS)$)
\[
m_j^{(\eta_{k})}(\ee^{(\eta_k)})
\leq
m_j^{(\eta_{k_0})}(\ee^{(\eta_k)})
<
\mol(\ee)
\leq
\mol(\ee^{(\eta_k)})
\leq
\mol^{(\eta_k)}(\ee^{(\eta_k)}).
\]
This is a contradiction,
since $m_i^{(\eta_{k})}(\ee^{(\eta_k)})=\mol^{(\eta_k)}(\ee^{(\eta_k)})$ for each $i\in \{0,1,\dots,n\}$.
\end{proof}

\begin{prop} \label{prop:mainJusc}
Let $n\in\NN$, $\nu_1,\ldots,\nu_n>0$,
$K$ be a monotone \eqref{cond:monotone} kernel function and $J$ be an  upper semicontinuous $n$-field function.
Then there exists an equioscillation point  $\ee\in Y$ with
\[
\mol(\ee)=\mul(\ee)=m(\oSS)=M(\oSS).
\]
Moreover, strict majorization does not hold on $Y$, i.e., there are no $\xx, \yy \in Y$
with $m_j(\xx)>m_j(\yy)$ for every $j\in \{0,1,\ldots,n\}$.
In particular, the equioscillation value is unique.
\end{prop}
\begin{proof}
	We can assume that $K$ is non-singular, otherwise we already know the statement from Theorem \ref{thm:mainold}.
	
By Proposition \ref{prop:eqexistsmm} there is an equioscillation point $\ee\in \oS$ with $\mol(\ee)=M(\oSS)$. By Proposition \ref{prop:nonmajorize} we have no strict majorization on $Y$, so by Lemma \ref{lem:eqmS} $\mol(\ee)=m(\oSS)$ follows.
\end{proof}

\section{Not upper semicontinuous field functions}\label{sec:furtherminimax}

In the following denote the \emph{upper semicontinuous regularization} of a function $\phi: [0,1]\to \uR$ by $\phi^*$, i.e.
\[
\phi^*(t):=\lim_{r\to 0} \sup_{s\in (t-r,t+r)\cap [0,1]}\phi(s).
\]
Then $\phi^*$ is an upper semicontinuous function,
moreover, it is the least upper semicontinuous function above $\phi$, see, e.g., p.~62 in \cite{Ransford}.

\begin{lemma}\label{lem:mbarinv} Let $K_j:[-1,1]\to \uR$ be extended continuous functions and $J:[0,1]\to \uR$ be an arbitrary upper bounded function.  Consider the corresponding generalized sum of translates function $F$ from \eqref{eq:Fsumgen}. Then the supremum  $\mol(\xx):=\sup_{[0,1]} F(\xx,\cdot)$ remains unchanged if $J$ is substituted by its upper semicontinuous regularization $J^*$.
\end{lemma}
\begin{proof} Obviously, replacing $J$ by $J^*$ furnishes $\mol^*(\xx):=\sup_{[0,1]}F^*(\xx,\cdot)$, where $F^*(\xx,\cdot)$ is the upper semicontinuous regularization of $F(\xx,\cdot)$. As the supremum of a function equals the supremum of its upper semicontinuous regularization, the statement is clear.
\end{proof}
Consider the upper semicontinuous regularization $J^*$ of $J$ and denote the corresponding sum of translate function and the corresponding quantities by $F^*$,  $M^*$ and $m^*$, respectively.

\begin{lem}\label{lem:mjinv} Let $K_j:[-1,1]\to \uR$ be extended continuous functions, singular in the sense that $K_j(0)=-\infty$, and let $J:[0,1]\to \uR$ be an arbitrary upper bounded function.  Consider the corresponding generalized sum of translates function $F$ from \eqref{eq:Fsumgen}. Then for each $j=0,1,\ldots,n$ the $j$th interval maximum function $m_j(\xx):=\sup_{[x_j,x_{j+1}]} F(\xx,\cdot)$ remains unchanged if $J$ is substituted by its upper semicontinuous regularization $J^*$. In particular, $\mv:\oS\to\uR$ remains unchanged under replacing $J$ by $J^*$.
\end{lem}

Observe that without singularity this statement does not remain valid. Indeed, if $K\equiv 0$ and $J=\log {\bf 1}_{[0,1/2)}$, then $J^*=\log {\bf 1}_{[0,1/2]}$, and for $n=1$ and $\xx=(1/2)$ the quantities $m_1(1/2)=-\infty$ and $m_1^*(1/2)=0$ are different.

\begin{proof}
	Consider the upper semicontinuous regularization $J^*$ of $J$ and denote the corresponding sum of translate function and the corresponding quantities by $F^*$,  $m_j^*$, respectively.	Let some $\xx\in \oS$ and an index $j\in \{0,1,\ldots,n\}$ be fixed arbitrarily. As any relative boundary point of $I_j(\xx)$ must be some node $x_i$ from the node system $\xx$, we already have $f(\xx,x_i)=-\infty$, and we get both for $J$ and for $J^*$ that $F(\xx,x_i)=-\infty$ and $F^*(\xx,x_i)=-\infty$. It follows that we can neglect boundary points, and $m_j(\xx)=\sup_{t \in \rint I_j(\xx)} F(\xx,t)$ and $m^*_j(\xx)=\sup_{t \in \rint I_j(\xx)} F^*(\xx,t)$ (with the convention $\sup\emptyset=-\infty$).

However, $F^*(\xx,t)|_{\rint I_j(\xx)}$, originally defined by $F^*(\xx,t):=f(\xx,t)+J^*(t)$, where $J^*$ is the (global) upper semicontinuous regularization of $J$, becomes the upper semicontinuous regularization of $F(\xx,t)$ on $\rint I_j(\xx)$. This is not necessarily true on sets which are not relatively open, but hold true on all relatively open subsets $A\subset [0,1]$. Indeed, for any point $t\in A$ there is a neighborhood $(t-\de,t+\de)\cap [0,1]\subset A$, hence $\limsup_{s \to t} J(s)=\limsup_{s\to t, s\in A} J(s)$, and the ``global'' definition of upper semicontinuous regularization coincides the restricted to $A$ definition. Thus we see that $m_j(\xx)$ is the supremum of a function $F(\xx,\cdot)|_{\rint I_j(\xx)}$ on $\rint I_j(\xx)$, and $m_j^*(\xx)$ is the supremum of its upper semicontinuous regularization on the same domain, therefore they match. The assertion is proved.
\end{proof}

In view of the above considerations with replacing the field function $J$ by its upper semicontinuous regularization $J^*$, we can drop the condition of upper semicontinuity from Theorem \ref{thm:mainold}.

\begin{prop}\label{prop:new32} Let $K$ be a singular \eqref{cond:infty} and monotone \eqref{cond:monotone} kernel function, and let $J$ be any field function.
Then $M(\oSS)=m(\oSS)$ and there exists some node system $\ww\in \oS$, with the three properties that it is an equioscillation point, it attains the simplex maximin and also it attains the simplex minimax: $\mul(\ww)=m(\oSS)=M(\oSS)=\mol(\ww)$.

Moreover, strict majorization $m_i(\xx)>m_i(\yy)$ for every $i=0,1,\ldots,n$ cannot hold for any $\xx, \yy \in Y$.
\end{prop}
\begin{proof} The assertion follows from Lemma \ref{lem:mjinv} (applied with $K_j:=\nu_jK$) and Theorem \ref{thm:mainold}.
\end{proof}

\begin{lemma}\label{lem:mmcsill} Let $K$ be a monotone \eqref{cond:monotone} kernel function, $\nu_i>0$ $(i=1,\ldots,n)$. Let $J$ be an arbitrary $n$-field function and let $J^*$ be its upper semicontinuous regularization. Denote by $m^*(\oSS)$ the corresponding maximin quantity  and $m_j^*$ the corresponding interval maximum functions for the case of the function $J^*$ as a field (note that, indeed, $J^*$ is an $n$-field function). Then $m(\oSS)= m^*(\oSS)$.
\end{lemma}
\begin{proof}
We may suppose that $K$ is non-singular, hence in view of  \eqref{cond:monotone} also finite valued, otherwise the statement is a direct consequence of Lemma \ref{lem:mjinv}. The inequality $m(\oSS)\le m^*(\oSS)$ is obvious, so we are to prove the converse inequality $m(\oSS)\ge m^*(\oSS)$.

Suppose for a contradiction that $m(\oSS)<m^*(\oSS)$, so the set $A:=\{\xx\in\oS:\mul^*(\xx)>m(\oSS)\}$ is non-empty.

For $\yy\in \oS$ let
\[
\I(\yy):=\{ i~:~ m_i(\yy)>m(\oSS)\} \qquad \text{and} \qquad \J(\yy):=\{j~:~ m_j(\yy)\le m(\oSS) \}.
\]
Note that $\J(\yy)$ is not empty, since $\mul(\yy)\leq m(\oSS)$.

Take now  an $\xx\in A$ such that the cardinality $\card\J (\xx)$ is minimal. Let $j\in \J(\xx)$. We will modify $\xx$ to $\xx'$ such that $m_j(\xx')>m(\oSS)$ while $\I(\xx)\subseteq \I(\xx')$. These yield  $\card\J(\xx')<\card\J(\xx)$, a contradiction to the minimality of $\card \J(\xx)$, proving the statement.

We have $m^*_j(\xx)\ge \mul^*(\xx) >m(\oSS)\geq m_j(\xx)$. Since a function and its upper semicontinuous regularization have the same supremum on each open set, we have
\[
\sup_{t\in\rint[x_j,x_{j+1}]}F^*(\xx,t) = \sup_{t\in \rint[x_j,x_{j+1}]}F(\xx,t).
\]

Therefore, $m_j^*(\xx)>m_j(\xx)$ is only possible if $F^*(\xx,x_j)=m^*_j(\xx)>m_j(\xx)$ or $F^*(\xx,x_{j+1})=m^*_j(\xx)>m_j(\xx)$; by symmetry, let us assume the first case of these. 
So in particular we must have $\limsup_{t\upto x_j}J(t)=J^*(x_j)> J(x_j)$.

So there is $\delta>0$ such that for every $\eta$ with $0<\eta<\delta$
\begin{equation}\label{eq:JJstar}
\sup_{s\in (x_j-\eta,x_j)}J(s)\geq J^*(x_j) > J(x_j) .
\end{equation}
Let $\veps>0$ be such that $ m^*_j(\xx)-2\veps>m(\oSS)$ and
\[
m_i(\xx)-\veps>m(\oSS)\quad\text{for each $i\in \I(\xx)$}.
\]
 By the uniform continuity of $f$, to the given $\ve>0$ we can take a $\delta_0\in (0,\delta)$ such that \eqref{puresumfuncontinuity} holds.

Let $\calX_j:=\{k\in \{1,\dots,n\}: x_k=x_j\}$.

Take $i\in \{0,1,\dots,n\}$ with $x_j\not\in I_i(\xx)$, if there is any. 
Consider the set
\[
A_{i,j,\xx}:=\{\zz\in \oS: z_k=z_j\not\in I_i(\xx)\text{ for each $k\in \calX_j$ and $z_k=x_k$ for each $k\not\in \calX_j$}\}.
\]
By Lemma \ref{lem:mjcontrest2} the function $m_i$ is continuous on the set $A_{i,j,\xx}$ and, of course,  $\xx\in A_{i,j,\xx}$.
So we can take $\al_i>0$, also with $\al_i<\min ( |x_j-x_i|, |x_j-x_{i+1}|)$, such that for $\zz\in A_{i,j,\xx}$ satisfying $\|\zz-\xx\|<\al_i$ we have
\[
m_i(\zz)\geq m_i(\xx)-\veps.
\]

Next we define a perturbation of $\xx$. Set $\al:=\min( \de_0, \min \{ \al_i ~:~ x_j\not\in I_i(\xx) \})$ and pick $x_j'\in (x_j-\al,x_j)$ such that
\[
J(x_j')\geq J^*(x_j)-\veps,
\]
which is possible in view of \eqref{eq:JJstar}. Now, we move all the nodes that coincide with $x_j$  to $x_j'$ and leave all other nodes unchanged, thus obtaining the new node system $\xx'$. Note that $\xx'$ is indeed a node system, i.e. the ordering of the nodes remain in the natural order, in view of the choice of the $\al_i$, and $\xx'\in A_{i,j,\xx}$ for every $i$ with $x_j\not \in I_i(\xx)$.

\medskip Now, let $i\in \{0,1,\dots,n\}$ be arbitrary.
If $x_j\not\in I_i(\xx)$, then $\xx'\in A_{i,j,\xx}$ and $\|\xx'-\xx\|<\al\le \al_i$, so by the above
\[
m_i(\xx')\geq m_i(\xx)-\veps.
\]
Thus for such $i$, we have $m_i(\xx')>m(\oSS)$ (i.e., $i\in \I(\xx')$) if $m_i(\xx)>m(\oSS)$ (i.e., if $i\in \I(\xx)$).

On the other hand, if $x_j\in I_i(\xx)$, then either $x_j=x_i$ or $x_j=x_{i+1}$, hence also $x_j'=x_i'$ or $x_j'=x_{i+1}'$, and then $x_j'\in I_i(\xx')$, too. Therefore,
\begin{align*}
m_i(\xx')&\geq F(\xx',x_j')=f(\xx',x_j')+J(x_j')\geq f(\xx,x_j)-\veps+J^*(x_j)-\veps\\
&=F^*(\xx,x_j)-2\veps=m^*_j(\xx)-2\veps>m(\oSS).
\end{align*}
In particular, choosing $i=j$ we see $m_j(\xx')>m(\oSS)$---i.e., $j\in \I(\xx')$. Further, from the argumentations in the cases $x_j\not\in I_i(\xx)$ and $x_j \in I_i(\xx)$ altogether also $ \I(\xx)\subseteq \I(\xx')$ follows.
The proof is complete.
\end{proof}

\bigskip

\begin{proof}[Proof of  Theorem \ref{thm:main}]
Consider the upper semicontinuous regularization $J^*$ of $J$ and denote the corresponding sum of translate function and the corresponding quantities by $F^*$,  $M^*$ and $m^*$, respectively. Then $m^*(\oSS)=M^*(\oSS)$ by Proposition \ref{prop:mainJusc} and $M^*(\oSS)=M(\oSS)$ by Lemma \ref{lem:mbarinv}. Altogether, we already have
\begin{equation}\label{eq:mmcsill}
m(\oSS)\leq m^*(\oSS)=M^*(\oSS)=M(\oSS).
\end{equation}
Furthermore, the equality $M(\oSS)=m(\oSS)$ follows from \eqref{eq:mmcsill} and Lemma \ref{lem:mmcsill}. Further, if $\ee$ is any equioscillation point (there may not be any in general!), then $M(\oSS)\le \mol(\ee)=\mul(\ee) \le m(\oSS)$ so one has equality everywhere. Also, there exists a point $\ww\in \oS$ with $\mol(\ww)=M(\oSS)$ in view of the continuity of $\mol$ on $\oS$, furnished by Lemma \ref{lem:mbarcont}. Moreover, if $J$ is upper semicontinuous, then according to Proposition \ref{prop:eqexistsmm} there exists an equioscillation point.

\bigskip
Regarding majorization, assume for a contradiction that strict majorization holds for some node systems $\xx,\yy\in Y$, that is, $m_j(\yy) > m_j(\xx)> -\infty$ for all $j=0,1,\ldots,n$. Take $\ve>0$ with
\[
\ve <\min\{m_j(\yy)-m_j(\xx): \  j=0,1,\ldots,n \}.
\]
Further, for all $j=0,1,\ldots,n$, there exists $u_j\in I_j(\yy)$ such that $F(\yy,u_j)> m_j(\yy)-\ve$ and $v_j\in I_j(\xx)$ such that $F(\xx,v_j)>-\infty$. Let us consider the sets $U:=\{u_j~:~ j=0,1,\ldots,n\}\subset [0,1]$ and $V:=\{v_j~:~ j=0,1,\ldots,n\}\subset [0,1]$. If $\# (U\cup V)= n+1$, then we are fine, but there can occur some coincidences among the $u_j$ and $v_k$, which can cause $\# (U\cup V)$ being smaller than $n+1$. To avoid arising complications, we can  take any subset $W\subset X^c$ with the weighted cardinality (with weights $1/2$ at endpoints $0, 1$) exceeding $n$. This is the exact condition for $J$ being an $n$-field function, thus such a set $W$ exists (even if it may not be disjoint from $U\cup V$). In any case, consider now the union $A:=U\cup V\cup W$.

Now we introduce the new upper bounded function
\[
\widetilde{J}(t):=\begin{cases}
J(t), & \text{ if } t\in A, \\
-\infty, &\text{ otherwise.}
\end{cases}
\]
Observe that $\widetilde{J}$ is upper semicontinuous, and it is an $n$-field function, too, as it clearly satisfies the finiteness condition for that category. Keeping the kernel $K$ and the coefficients  $\nu_i,~i=1,\ldots,n$, but taking $\widetilde{J}$ in place of $J$ we get the respective sum of translates function $\widetilde{F}$ and interval maxima $\widetilde{m}_j$, respectively. It is clear that $\widetilde{J}\le J$, hence $\widetilde{F}\le F$ and $\widetilde{m}_j\le m_j$, too. By the choice of $v_0,v_1,\dots, v_n$ we have $\widetilde m_j(\xx)>-\infty$ for every $j\in\{0,1,\dots,n\}$, i.e., $\xx$ belongs to the regularity set $\widetilde Y$ for the modified situation.

Finally,
\[
m_j(\yy)-\ve < F(\yy,u_j) = \widetilde{F}(\yy,u_j) \le \widetilde{m}_j(\yy),
\]
therefore
\[
\widetilde{m}_j(\xx) \le m_j(\xx) < m_j(\yy)-\ve \le \widetilde{m}_j(\yy),
\]
which shows that we have strict majorization with the upper semicontinuous $n$-field function $\widetilde{J}$, contradicting Proposition \ref{prop:nonmajorize}.
Theorem \ref{thm:main} is proved.
\end{proof}

\section{Concluding remarks and open problems}\label{sec:conclusion}

With the present study we have arrived at a considerably complete description
 of Fenton type minimax theorems for the interval.
Therefore, it is in order to comment on the developments thus far
and the many directions where further work may be interesting.

Our above Theorem \ref{thm:main} formulates a Fenton type result under essentially minimal assumptions,
once one accepts the concavity assumption on $K$.
Neither singularity, nor the cusp condition \eqref{cond:inftyprime} of Fenton is assumed here,
and neither strict concavity (or monotonicity), nor any smoothness assumptions or
conditions involving derivatives are used.
Furthermore, the field function can be almost arbitrary, satisfying only the necessary requirement
so that $F(\xx,\cdot)$ becomes non-constant $-\infty$ for every $\xx$.
Beside convexity of $K$,
 the only extra assumption is its  monotonicity, which, on the other hand,
was used many times throughout the argumentations.
It was already shown in \cite{Minimax} that the monotonicity condition cannot be just dropped:
The assertions badly fail to hold without it.
However, the role of monotonicity is still a bit mysterious,
because in \cite{TLMS2018} we did not require it for the periodic (i.e., torus) case,
while in \cite{Homeo} concerning the case of singular kernels on the interval,
we could prove a homeomorphism theorem under a weaker assumption
\[
K'(t)-K'(t-1)\geq c> 0 \quad \text{for a.e.{} $t\in[0,1]$},
\]
called ``periodized $c$-monotonicity''.
At the the same time in the periodic case (which would correspond to $c=0$)
also some singularity-type assumptions were needed for the field $J$.

\medskip

For not necessarily singular kernels the argumentation had to follow
a different route than the one for singular kernels. Indeed, a crucial difference
is that in case of singular kernels a surprisingly general homeomorphism theorem was found.
This result, Theorem 2.1 of \cite{Homeo}, however, must necessarily fail in case the kernels are non-singular, as is discussed in Example 8.2 of the mentioned latter paper.

The yields of the homeomorphism theorem, so useful in case of singular kernels, are thus irrelevant
for the minimax results here in the case of non-singular kernels.
However, it would be interesting to clarify the following matters
(this has been done for singular $K$ in \cite{Homeo}).
Do we still have that, under some suitable assumptions, the sum of translate function,
and even the functions $m_j$ and the maximum vector $\mv$ are locally
Lipschitz functions (on an appropriate domain)?
Note that in full generality the $m_j$ are not even continuous, as is seen from the above mentioned Example 8.2 of \cite{Homeo}, so any affirmative answer could be formulated on some appropriate domain and under some suitable assumptions at best. Further, do we have local Lipschitz continuity for the functions
$\mol$ and $\mul$, too? If yes, is it still true that the a.e.{} existing
derivative provides a non-singular Clarke derivative (generalized Jacobian)
for the difference function $\Phi:=(m_1-m_0,\ldots,m_{n+1}-m_n)$,
as happens for singular kernels, see Section 6 in \cite{Homeo}?
If yes, then $\Phi$ is a locally invertible locally Lipschitz mapping.
Is it  a global homeomorphism between $Y$ and its image $\Phi(Y)\subsetneqq \RR^n$?
At first glance, there is quite a difficulty here. Indeed, in the singular case
we could apply the well-known topology result usually attributed to Hadamard and, according to our search, recorded first by Eilenberg \cite{Eilenberg1935} in the thirties (and recurred several times since then, see the references in \cite{Homeo}), which says that if a mapping is a local homeomorphism and
 is unbounded as we approach the boundary, 
then it is a global homeomorphism, too. However, in the case $K$ is non-singular, even if local invertibility may be proven, the properness of the difference mapping $\Phi$ is not at all clear.

\medskip

As said, for non-singular kernels there is no homeomorphism theorem,
a key ingredient to prove the failure of majorization $m_j(\xx)\leq m_j(\yy)$
for every $j\in \{0,1,\dots,n\}$ for $\xx,\yy\in Y$, called \emph{intertwining};
this  is Theorem 4.1 of \cite{Minimax}.
However, on a completely different route---by elementary but involved arguments---in \cite{Ural}
we could arrive at such intertwining results even for non-singular
(but strictly concave and monotone) kernels, unfortunately only if $n\in \{1,2,3\}$.
In \cite{Ural} we  thus also formulated the conjecture that
the same phenomenon of intertwining occurs for larger $n$, too---a somewhat mysterious challenge which we could not progress with any further.
A possible approach is to seek the ``next best thing'' and replace the global homeomorphism by global injectivity. We have to admit that, although we have \emph{local} injectivity,  global injectivity of $\Phi$, appears to be totally out of grasp.
Nevertheless, to us it seems plausible that in cases when global injectivity of $\Phi$ remains true---and maybe other favorable circumstances, such as continuity of the functions $m_j$, hold---then also the intertwining result extends.

\medskip
The weaker property, the exclusion of \emph{strict} majorization,
is completely settled in the present paper, as a part of the main result Theorem \ref{thm:main}.
This also indicates the plausibility of the conjecture put forward in \cite{Ural}.
Given that we have the statement for small $n$, an inductive argument would seem natural,
but as the number of cases in the combinatorics of the possible inequalities
between the nodes $x_i, y_i$ grow super-exponentially, the case when $n>3$ appeared intractable.
Of course, there can be also some other approach, proving the statement for sufficiently nice,
say continuous fields, and then one could try to argue for general fields as we did in the proof of
Theorem \ref{thm:main} at the end of Section \ref{sec:furtherminimax}.

\medskip

In this paper we looked for the possibly least restrictive conditions
under which the \emph{original} result of Fenton may be extended.
We  thus have dealt with intervals, normalized to $[0,1]$ and $[-1,1]$.
However, other setups are also important. Analogous questions for the
torus $\TT:=\RR/\ZZ$ (periodic/unit circle) were already brought up
by Ambrus, Ball, Erd\'elyi \cite{AmbrusBallErdelyi} in connection with
the so-called strong polarization constant problem---we have accounted for
the development in this direction shortly in the Introduction and in more detail in \cite{Minimax}.
Even if the condition of monotonicity does not play a role in the torus setup, the corresponding
Fenton theory still stays much less developed in this case. In our work \cite{TLMS2018} we considered the situation when also $J$ (called $K_0$ there) is concave, but later developments in the interval case, relaxing considerably the assumptions on $J$, were not followed on $\TT$. Therefore, the direction of relaxing conditions on $J$ in the periodic case remains the subject of future investigations.

\medskip

An interesting observation, however trivial, is that an essential difference
between the interval and torus setup is that for the interval certain parts of
the defining domain of the kernel functions $K$ are translated in, and other parts
are translated out from the space (i.e., $[0,1]$) where we consider the maxima and the minimax questions.
In the torus setup all the translates live on the whole $\TT$,
 where also the maxima etc. are considered.
That difference seems to classify two cases of reasonable
setups---our feeling is that this may be connected with the necessity
or dispensability of extra monotonicity conditions.
Note that interpreting the torus setup by taking $K$ periodic $\mod 1$,
it is clear that periodicity and the monotonicity condition \eqref{cond:monotone}
can both hold only in the trivial, i.e. $K\equiv\text{constant}$ cases.

\medskip

As for other meaningful setups, the cases of the real line and the positive halfline must be mentioned.
These are important because of their numerous applications; for example, much of weighted
approximation theory and weighted potential theory were developed right for them.
(In fact, Fenton's original application, in the theory of entire functions,
would also fit more naturally to the situation of the infinite halfline.)
In both cases, the kernel  $K$ is  defined on $\RR$ being concave on  $(-\infty,0)$ and $(0,\infty)$.
The differences are in the definition of $J$, and in the possible locations of nodes in $\RR$
or in $[0,\infty)$, respectively. The translating in -- translating out aspect makes the halfline similar to the interval case, and $\RR$ similar to the torus case.

Also, for infinite intervals some kind of admissibility condition is required for $J$,
since otherwise the suprema on the unbounded intervals may be infinite. Such admissibility conditions are also present in weighted logarithmic potential theory on the real line resp.{} halfline. At this point we need to have a closer look at our normalization conditions. In logarithmic potential theory very often the total weight of translates---the total measure with which we consider potentials---is normalized to 1. If so, also the weight function $w$, or its logarithm $J:=\log w$ can be considered with power (weight) 1. If on the other hand we consider degree $n$ polynomials without normalizing them, then the total measure of translates becomes $n$ and then also $w^n$ becomes the natural counterpart. This motivates our normalization with considering the constant $\nu:=\sum_i \nu_i$, and assume, as a natural admissibility condition, that we have
\begin{equation}\label{eq:admissible}
\lim_{|t|\to \infty} \nu K(t)+J(t)=-\infty.
\end{equation}
In  the case of the halfline $[0,\infty)$ this should be modified by writing ``$t\to +\infty$''  only.

This admissibility condition should be compared to the ones
generally used in logarithmic potential theory of weighted approximation,
for example to (iii) in Definition I.1.1 for an admissible weight $w:\Sigma\to [0,\infty)$,
$\Sigma\subset \CC$ in the book \cite{Saff-Totik} (the relation to our framework is $\log w=J$).

The situation there with an arbitrary, closed $\Sigma$, where the weight is defined, and corresponding questions of minimax etc.{} are posed, is seemingly more general than ours restricting to intervals.
However, for absolute value maximization etc.~ one can easily reconfigure
the system of conditions by extending $w$ as identically $0$ ($J\equiv -\infty$) outside $\Sigma$
(a modification not spoiling the upper semicontinuity condition postulated in (ii)
in the mentioned Definition I.1.1 in view of the assumption that $\Sigma$ is closed).

So, the logarithmic potential theory setup considers $\Sigma^c$ as a set
where $w$ is not defined, and invokes still another set $\Sigma_0 \subset \Sigma$,
where it may not vanish, and which set must have positive inner capacity.
The distinction is, however, inessential: the setup corresponds to allowing $J=\log w$
being $-\infty$ on $\Sigma^c \cup (\Sigma\setminus \Sigma_0)$, i.e. altogether on $\Sigma_0^c$.
So our setup, when we allow $X:=J^{-1}(\{-\infty\})$ to be any set which leaves
more than $n$ points in its complement, is already much more general
(as we do not assume positive capacity) for sets $\Sigma_0\subseteq \RR$. The same way, we are more general \emph{in our basic setup} when skipping the condition of upper semicontinuity of $J$ or $w$, assumed in (ii) of the said definition
(and throughout logarithmic potential theory). Needless to say, $\log|\cdot|$ is just one particular example of concave (and also singular and strictly monotone) kernel functions.

Altogether, our setup is more general than that of weighted approximation theory
and logarithmic potential theory, so we can expect general results which
have direct consequences in various special settings of the latter, like e.g. Chebyshev
constants of perfect sets, Bojanov type approximation theorems etc.
On the other hand it is clear that Fenton theory, at least in its current degree of development,
cannot deal with the most important questions of asymptotic nature
when some limits are taken with respect to the number of translates etc. These asymptotic questions form the most important goals of logarithmic potential theory, and it is natural that for handling them some additional conditions need to be required. We have seen that Fenton's approach can give surprisingly precise results in a greater generality---for the minimax, equioscillation and majorization questions studied for given $n$ and coefficients $\nu_i$. However, we have not addressed asymptotic questions of potential theory with this approach. There is such a chapter of general potential theory, see, e.g., \cite{choquet:1953}, \cite{choquet/deny:1956}, \cite{ohtsuka:1961}, \cite{fuglede:1960}, \cite{zorii-lodz,zorii-ukr}, \cite{FN08}, \cite{FR06c}, where even without concavity much of potential theory is developed,
but it remains to future work to connect these and our Fenton setup.

\medskip

Let us point out just one natural question, which might be of the same
crucial importance in the Fenton theory of infinite intervals
as it plays in logarithmic potential theory.
Namely, in the latter the existence of a so-called
Mhaskar-Rakhmanov-Saff 
number  $a_n:=a_n(w)$ is proved, with the property that once we consider a polynomial $p$ of degree
not exceeding $n$ with the given admissible weight $w$ (or $w^n$),
then $\|p\|_w:=\|p w\|_{\infty}$ (or $\| pw^n\|_{\infty}$) equals the supremum norm on $[-a_n,a_n]$.
(This is generally expressed by saying that a weighted polynomial ``lives'' on the
Mhaskar-Rakhmanov-Saff interval $[-a_n,a_n]$.) Taking logarithms of absolute values and
specifying $x_i$ as the zeros of $p$, with taking $\nu_i=1$
this translates to $\sup_{t\in \RR} (\log|p(t)|+\log w(t))= \sup_{t\in [-a_n,a_n]}(\log|p(t)|+\log w(t))$.
Our question is if in general, for arbitrary concave kernels and possibly also very general fields $J$ satisfying \eqref{eq:admissible} this remains in effect?

%

\medskip
Last, but not least, let us mention a question, which goes in a different direction.
Already with the general homeomorphism theorem for singular kernels,
we could deal with \emph{different} kernels $K_j$, not just \emph{different multiples}
$\nu_j K$ of the same base kernel function $K$.
It is hard to formulate a real life analogy, like potential theory, for this case,
but maybe in physical systems where different forces and laws, such as strong forces and gravity etc. superpose,
some heuristics can be found.
At any rate, the question is challenging: How much of the Fenton theory goes through,
and under which (possibly minimally restrictive) conditions?
Note that at several instances we have already formulated general results,
sometimes even dropping concavity, but these partial results did not yet compile into a unified theory.
In particular, minimax and maximin results as well as equioscillation property
of extremal points, can possibly hold true, at least for suitably tailored fields.
In our current work we were aiming at allowing the possibly most general field functions;
but it remains an open problem to clarify whether the basic findings of Fenton's theory can be established also for different kernels under some suitable assumptions.

%
\providecommand{\bysame}{\leavevmode\hbox to3em{\hrulefill}\thinspace}
\providecommand{\MR}{\relax\ifhmode\unskip\space\fi MR }
\providecommand{\MRhref}[2]{%
  \href{http://www.ams.org/mathscinet-getitem?mr=#1}{#2}
}
\providecommand{\href}[2]{#2}

\medskip

\noindent
\hspace*{5mm}
\begin{minipage}{\textwidth}
\noindent
\hspace*{-5mm}Bálint Farkas\\
 School of Mathematics and Natural Sciences,\\
 University of Wuppertal\\
  Gau{\ss}stra{\ss}e 20\\
 42119 Wuppertal, Germany\\
\end{minipage}

\medskip

\noindent
\hspace*{5mm}
\begin{minipage}{\textwidth}
\noindent
\hspace*{-5mm}Béla Nagy\\
 Department of Analysis,\\
 Bolyai Institute, University of Szeged\\
 Aradi vértanuk tere 1\\
  6720 Szeged, Hungary\\
\end{minipage}

\medskip

\noindent
\hspace*{5mm}
\begin{minipage}{\textwidth}
\noindent
\hspace*{-5mm}
Szilárd Gy.{} Révész\\
 Alfréd Rényi Institute of Mathematics\\
 Reáltanoda utca 13-15\\
 1053 Budapest, Hungary \\
\end{minipage}

\end{document}